\documentclass[a4paper,oneside,11pt]{amsart}
\usepackage[latin1]{inputenc}
\usepackage[T1]{fontenc}
\usepackage{amssymb,amsfonts,amsmath,amsthm}
\usepackage{paralist}
\usepackage{chngpage}

\usepackage[colorlinks=true, linkcolor=blue, citecolor=red, urlcolor=blue]{hyperref}
\usepackage[stretch=10,shrink=10,step=4,kerning=true,protrusion=true,final]{microtype}

\newtheorem{theorem}{Theorem}
\newtheorem{proposition}[theorem]{Proposition}
\newtheorem{corollary}[theorem]{Corollary}
\newtheorem{lemma}[theorem]{Lemma}

\theoremstyle{definition}
\newtheorem{definition}[theorem]{Definition}

\theoremstyle{remark}
\newtheorem{remark}[theorem]{Remark}
\newtheorem{example}[theorem]{Example}

\newcommand{\ie}{i.e.\ }
\newcommand{\eg}{e.g.\ }
\newcommand{\resp}{resp.\ }
\newcommand{\PP}{\mathbf{P}}
\newcommand{\CC}{\mathbf{C}}
\newcommand{\bH}{\mathbf{H}}
\newcommand{\RR}{\mathbf{R}}

\newcommand{\NN}{\mathbf{N}}
\newcommand{\KK}{\mathbf{K}}
\newcommand{\GL}{\mathrm{GL}}

\newcommand{\PSL}{\mathrm{PSL}}
\newcommand{\SO}{\mathrm{SO}}
\newcommand{\OO}{\mathrm{O}}

\newcommand{\U}{\mathrm{U}}
\newcommand{\Sp}{\mathrm{Sp}}

\newcommand{\g}{\mathfrak{g}}
\newcommand{\aaa}{\mathfrak{a}}

\newcommand{\Hom}{\mathrm{Hom}}
\newcommand{\Aut}{\mathrm{Aut}}
\newcommand{\Diag}{\mathrm{Diag}}
\newcommand{\HH}{\mathbb{H}}
\newcommand{\overhatHpq}{\overline{\mathbb{H}_{\mathbf{K}}^{p,q}}\hspace{-0.54cm}\hat{\vphantom{\mathbb{H}}}\hspace{0.54cm}}

\newcommand{\ad}{\mathrm{ad}}

\newcommand{\F}{\mathcal{F}}
\newcommand{\specialK}{\mathcal{K}}

\newcommand*{\longhookrightarrow}{\ensuremath{\lhook\joinrel\relbar\joinrel\rightarrow}}

\title[Compactification of certain Clifford--Klein forms]{Compactification of certain Clifford--Klein forms of reductive homogeneous spaces}

\author[F.\ Gu\'eritaud]{Fran\c{c}ois Gu\'eritaud}
\address{CNRS and Universit\'e Lille 1, Laboratoire Paul Painlev\'e, 59655 Villeneuve d'Ascq Cedex, France 
\newline Wolfgang-Pauli Institute, University of Vienna, CNRS-UMI 2842, Austria}\email{francois.gueritaud@math.univ-lille1.fr}

\author[O.\ Guichard]{Olivier Guichard}
\address{Universit\'e de Strasbourg, IRMA, 7 rue  Descartes, 67000 Strasbourg, France}
\email{olivier.guichard@math.unistra.fr}

\author[F.\ Kassel]{Fanny Kassel}
\address{CNRS and Universit\'e Lille 1, Laboratoire Paul Painlev\'e, 59655 Villeneuve d'Ascq Cedex, France}
\email{fanny.kassel@math.univ-lille1.fr}

\author[A.\ Wienhard]{Anna Wienhard}
\address{Ruprecht-Karls Universit\"at Heidelberg, Mathematisches Institut, Im Neuenheimer Feld~288, 69120 Heidelberg, Germany
\newline HITS gGmbH, Heidelberg Institute for Theoretical Studies, Schloss-Wolfsbrun\-nenweg 35, 69118 Heidelberg, Germany}
\email{wienhard@uni-heidelberg.de}

\thanks{FG and FK were partially supported by the Agence Nationale de la Recherche under the grant DiscGroup (ANR-11-BS01-013) and through the Labex CEMPI (ANR-11-LABX-0007-01).
AW was partially supported by the National Science Foundation under agreement DMS-1536017, by the Sloan Foundation, by the Deutsche Forschungsgemeinschaft, and by the European Research Council under ERC-Consolidator grant 614733.
Part of this work was carried out while OG, FK, and AW were in residence at the MSRI in Berkeley, California, supported by the National Science Foundation under grant 0932078~000.
The authors also acknowledge support from U.S. National Science Foundation grants DMS 1107452, 1107263, 1107367 ``RNMS: GEometric structures And Representation varieties'' (the GEAR Network).}

\begin{document}

\numberwithin{theorem}{section}
\numberwithin{equation}{section}

\begin{abstract}
We describe smooth compactifications of certain families of reductive homogeneous spaces such as group manifolds for classical Lie groups, or pseudo-Riemannian analogues of real hyperbolic spaces and their complex and quaternionic counterparts.
We deduce compactifications of Clifford--Klein forms of these homogeneous spaces, namely quotients by discrete groups $\Gamma$ acting properly discontinuously, in the case that $\Gamma$ is word hyperbolic and acts via an Anosov representation.
In particular, these Clifford--Klein forms are topologically tame.
\end{abstract}

\maketitle

\vspace{-0.5cm}

\section{Introduction}

The goal of this note  is two-fold. 
First, we describe compactifications of certain families of reductive homogeneous spaces $G/H$ by embedding $G$ into a larger group $G'$ and realizing $G/H$ as a $G$-orbit in a flag manifold of~$G'$.
These homogeneous spaces include:
\begin{itemize}
  \item group manifolds for classical Lie groups (Theorems \ref{thm:compactif-Aut(b)} and~\ref{thm:compactif-GL}),
  \item pseudo-Riemannian analogues of real hyperbolic spaces and their complex and quaternionic counterparts (Theorem~\ref{thm:Hpq}.\eqref{item:compactif-Hpq}),
  \item certain affine symmetric spaces such as $\OO(2p,2q)/\U(p,q)$ or $\U(2p,2q)/\Sp(p,q)$ (Proposition~\ref{prop:compactification_general}.\eqref{item:compactif_general_G/H}),
  \item other reductive homogeneous spaces (Proposition~\ref{prop:compactification_general_nonsym}.\eqref{item:compactif_general_G/H}).
\end{itemize}
Second, we use these compactifications and a construction of domains of discontinuity from \cite{Guichard_Wienhard_DoD} to compactify Clifford--Klein forms of $G/H$, \ie quotient manifolds $\Gamma\backslash G/H$, in the case that $\Gamma$ is a word hyperbolic group whose action on $G/H$ is given by an Anosov representation $\rho : \Gamma\rightarrow G\hookrightarrow G'$.
We deduce that these Clifford--Klein forms are topologically tame.

Anosov representations (see Section~\ref{subsec:def-Ano}) were introduced in \cite{Labourie_anosov}.
They provide a rich class of quasi-isometric embeddings of word hyperbolic groups into reductive Lie groups with remarkable properties, generalizing convex cocompact representations to higher real rank \cite{Labourie_anosov,Guichard_Wienhard_DoD, KapovichLeebPorti, KapovichLeebPorti14, KapovichLeebPorti14_2, GGKW_anosov}.
Examples include:
\begin{enumerate}[(a)]
  \item\label{item:2} The inclusion of convex cocompact subgroups in real semisimple Lie groups of real rank~$1$ \cite{Labourie_anosov, Guichard_Wienhard_DoD};
  \item Representations of surface groups into split real semisimple Lie groups that belong to the Hitchin component \cite{Labourie_anosov, Fock_Goncharov, Guichard_Wienhard_DoD};
  \item Maximal representations of surface groups into semisimple Lie groups of Hermitian type \cite{Burger_Iozzi_Labourie_Wienhard, Burger_Iozzi_Wienhard_anosov, Guichard_Wienhard_DoD};
  \item The inclusion of quasi-Fuchsian subgroups in $\SO(2,d)$ \cite{Barbot_Merigot_fusionPubli, BarbotDefAdSQF};
  \item Holonomies of compact convex $\RR\PP^n$-manifolds whose fundamental group is word hyperbolic \cite{Benoist_CD1}.
\end{enumerate}

\subsection{Compactifying group manifolds} \label{subsec:intro-compact-G}

Any real reductive Lie group $G$ can be seen as an affine symmetric space $(G\times G)/\Diag(G)$ under the action of $G\times G$ by left and right multiplication.
We call $G$ with this structure a \emph{group manifold}.
We describe a smooth compactification of the group manifold $G$ when $G$ is a classical group.
This compactification is very elementary, in particular when $G$ is the automorphism group of a nondegenerate bilinear form. 
It shares some common features with the so-called \emph{wonderful compactifications} of algebraic groups over an algebraically closed field constructed by De Concini and Procesi \cite{DeConcini_Procesi_83} or Luna and Vust \cite{Luna_Vust}, as well as with the compactifications constructed by Neretin \cite{Neretin98, Neretin03}.
After completing this note, we learned that this compactification had first been discovered by He \cite{He_02}; we still include our original self-contained description for the reader's convenience.

We first consider the case that $G$ is $\OO(p,q)$, $\OO(m,\CC)$, $\Sp(2n,\RR)$, $\Sp(2n,\CC)$, $\U(p,q)$, $\Sp(p,q)$, or $\OO^*(2m)$.
In other words, $G = \Aut_{\KK}(b)$ is the group of $\KK$-linear automorphisms of a  nondegenerate $\RR$-bilinear form $b : V \otimes_{\RR} V \to \KK$ on a $\KK$-vector space~$V$, for $\KK=\RR$, $\CC$, or the ring $\bH$ of quaternions. We assume that $b$ is \(\KK\)-linear in the second variable, and that $b$ is symmetric or antisymmetric (if $\KK=\RR$ or~$\CC$), or Hermitian or anti-Hermitian (if $\KK=\CC$ or~$\bH$).
We describe a smooth compactification of $G = \Aut_{\KK}(b)$ by embedding it into the compact space of maximal $(b\oplus -b)$-isotropic $\KK$-subspaces of $(V\oplus V, b \oplus -b)$.
Let $n\in\NN$ be the real rank of $G=\Aut_{\KK}(b)$ and $N =\dim_\KK(V)\geq 2n$ the real rank of $\Aut_{\KK}(b\oplus -b)$.
For any $0\leq i\leq n$, let $\F_i(b) = \F_i(-b)$ be the space of $i$-dimensional $b$-isotropic subspaces of~$V$; it is a smooth manifold with a transitive action of~$G$.
We use similar notation for $(V\oplus\nolinebreak V, b\oplus -b)$, with $0\leq i\leq N$.
For any subspace $W$ of $V\oplus V$, we set 
\begin{equation} \label{eqn:pi}
\pi(W) := \big(W \cap (V \oplus \{ 0\}), W \cap (\{ 0\} \oplus V)\big).
\end{equation}
This defines a map $\pi : \F_N(b\oplus -b) \rightarrow (\bigcup_{i=0}^n \F_i(b)) \times (\bigcup_{i=0}^n \F_i(-b))$.

\begin{theorem} \label{thm:compactif-Aut(b)}
 Let $G = \Aut_{\KK}(b)$ be as above.
 The space $X = \F_N(b\oplus -b)$ of maximal $(b\oplus -b)$-isotropic $\KK$-subspaces of $V \oplus V$ is a smooth compactification of the group manifold $(G\times G)/\Diag(G)$ with the following properties: 
\begin{enumerate}
   \item $X$ is a real analytic manifold (in fact complex analytic if $\KK=\CC$ and $b$ is symmetric or antisymmetric).
   Under the action of a maximal compact subgroup of $\Aut_{\KK}(b\oplus -b)$, it identifies with a Riemannian symmetric space of the compact type, given explicitly in Table~\ref{table1}.
    \item The $(G \times G)$-orbits in~$X$ are the submanifolds \( \mathcal{U}_i  := \pi^{-1}(\F_i(b) \times \F_i(-b)) \) for $0\leq\nolinebreak i\leq\nolinebreak n$, of dimension $\dim_{\RR}(\mathcal{U}_i) = \dim_{\RR}(G) - i^2 \dim_\RR(\KK)$. 
    The closure of $\mathcal{U}_i$ in $X$ is $\bigcup_{j\geq i} \mathcal{U}_j$.
   \item For $0\leq i\leq n$, the map $\pi$ defines a fibration of $\mathcal{U}_i$ over $\F_i(b) \times\nolinebreak \F_i(-b)$~with fibers isomorphic to $(H_i\times H_i)/\Diag(H_i)$, where $H_i = \Aut_{\KK}(b_{V_i})$ is the automorphism group of the form $b_{V_i}$ induced by $b$ on $V_i^{\perp_b}/V_i$ for some $V_i\in\F_i(b)$.
\end{enumerate}
In particular, \(\mathcal{U}_0\) is the unique open\,\((G\times G)\)-orbit and it identifies with \((G\times G) / \Diag(G)\).
\end{theorem}

\begin{table}[h!]
\centering
\begin{adjustwidth}{-1.2in}{-1.2in}
\centering
\begin{tabular}{|c|c|c|c|}
\hline
$G$ & $n$ & $N$ & $X$ as a Riemannian symmetric space \tabularnewline
\hline
$\OO(p,q)$ & $\min(p,q)$ & $p+q$ & $(\OO(p+q)\times\OO(p+q))/\Diag(\OO(p+q))$ \tabularnewline
$\U(p,q)$ & $\min(p,q)$ & $p+q$ & $(\U(p+q)\times\U(p+q))/\Diag(\U(p+q))$ \tabularnewline
$\Sp(p,q)$ & $\min(p,q)$ & $p+q$ & $(\Sp(p+q)\times\Sp(p+q))/\Diag(\Sp(p+q))$ \tabularnewline
$\Sp(2n,\RR)$ & $n$ & $2n$ & $\U(2n)/\OO(2n)$ \tabularnewline
$\Sp(2n,\CC)$ & $n$ & $2n$ & $\Sp(2n)/\U(2n)$ \tabularnewline
$\OO(m,\CC)$ & $\lfloor\frac{m}{2}\rfloor$ & $m$ & $\OO(2m)/\U(m)$ \tabularnewline
$\OO^*(2m)$ & $\lfloor\frac{m}{2}\rfloor$ & $m$ & $\U(2m)/\Sp(m)$ \tabularnewline
\hline
\end{tabular}
\end{adjustwidth}
\vspace{0.2cm}
\caption{The compactification $X$ of Theorem~\ref{thm:compactif-Aut(b)}.}
\label{table1}
\end{table}

\begin{remark}
For $G=\OO(p,q)$, $\U(p,q)$, or $\Sp(p,q)$, the compactification $X$ identifies with the group manifold $(G_c\times G_c)/\Diag(G_c)$ where $G_c$ is the compact real form of a complexification of~$G$.
For $G=\OO(n,1)$, the embedding of $G$ into $G_c = \OO(n+1)$ lifts the embedding of
$\HH^n_{\RR}\sqcup\HH^n_{\RR} = \OO(n,1)/\OO(n)$ into $\mathbb{S}^n_{\RR} = \OO(n+1)/\OO(n)$
with image the complement of the equatorial sphere $\mathbb{S}^{n-1}_{\RR}$.
\end{remark}

Similar compactifications are constructed for general linear groups $\GL_{\KK}(V)$ in Theorem~\ref{thm:compactif-GL} below.

\subsection{Compactifying Clifford--Klein forms of group manifolds}

\hspace{-0.3cm} Let $G = \Aut_{\KK}(b)$ be as above.
For any discrete group~$\Gamma$ and any representation $\rho :\nolinebreak \Gamma \to G$ with discrete image and finite kernel, the action of~$\Gamma$ on~$G$ via left multiplication by~$\rho$ is properly discontinuous.
The quotient $\rho(\Gamma)\backslash G$ is an orbifold, in general noncompact.
Suppose that $\Gamma$ is word hyperbolic and $\rho$ is $P_1(b)$-Anosov (see Section~\ref{sec:rem-Anosov} for definitions), where $P_1(b)$ is the stabilizer in~$G$ of a $b$-isotropic line.
Considering a suitable subset of the compactification of~$G$ described in Theorem~\ref{thm:compactif-Aut(b)}, we construct a compactification of $\rho(\Gamma)\backslash G$ which is an orbifold, or if $\Gamma$ is torsion-free, a smooth manifold.

\begin{theorem} \label{thm:intro_comp_quotient}
Let $\Gamma$ be a word hyperbolic group and $\rho: \Gamma \to G = \Aut_{\KK}(b)$ a $P_1(b)$-Anosov representation with boundary map $\xi: \partial_\infty \Gamma \to \F_1(b)$.
For any $0\leq i\leq n$, let $\specialK^i_{\xi}$ be the subset of $\F_i(b)$ consisting of subspaces $W$ containing $\xi(\eta)$ for some $\eta \in \partial_\infty \Gamma$, and let $\mathcal{U}^\xi_i$ be the complement in~$\mathcal{U}_i$ of $\pi^{-1}(\specialK^i_{\xi} \times \F_i(-b))$, where $\pi$ is the projection of \eqref{eqn:pi}.
Then 
$\rho(\Gamma) \times \{ e\}  \subset \Aut_{\KK}(b) \times \Aut_{\KK}(b)$ acts properly discontinuously and cocompactly on the open subset 
\[ \Omega := \bigcup_{i=0}^n\, \mathcal{U}^\xi_i \]
of $\F_N(b\oplus -b)$.
The quotient orbifold $(\rho(\Gamma) \times \{ e\}) \backslash \Omega$ is a compactification of
\[ \rho(\Gamma) \backslash G \,\simeq\, (\rho(\Gamma) \times \{ e\}) \backslash (G \times G)/\Diag(G). \]
If $\Gamma$ is torsion-free, then this compactification is a smooth manifold. 
\end{theorem}

Similarly, we compactify quotients of $G = \Aut_{\KK}(b)$ by a word hyperbolic group $\Gamma$ when the action is given by any $P_1(b\oplus -b)$-Anosov representation
\[ \rho : \Gamma \longrightarrow \Aut_{\KK}(b) \times \Aut_{\KK}(b) \subset \Aut_{\KK}(b\oplus -b), \]
\ie we allow $\Gamma$ to act simultaneously by left and right multiplication instead of just left multiplication: this is the object of Theorem~\ref{thm:compactification}.
When $G = \Aut_{\KK}(b)$ has real rank~$1$, all properly discontinuous actions on $(G\times G)/\Diag(G)$ via quasi-isometric embeddings are of this form \cite[Th.\,6.5]{GGKW_anosov}; the case $G = \OO(p,1)$ for $p=2$ or~$3$ is particularly interesting from a geometric point of view (Remark~\ref{rem:proper-implies-Ano-rk1}).
  
\begin{remark} \label{rem:not-smooth}
Let $G$ be an arbitrary real reductive Lie group and $P$ a parabolic subgroup.
Composing a $P$-Anosov representation $\rho:\Gamma \to G$ with an appropriate linear representation $\tau: G \to \Aut_\KK(b)$, we obtain a $P_1(b)$-Anosov representation $\tau \circ \rho: \Gamma \to \Aut_\KK(b)$.
Theorem~\ref{thm:intro_comp_quotient} can then be applied to give a compactification of $\rho(\Gamma)\backslash G$: see Corollary~\ref{cor:general_comp} for a precise statement.
\end{remark}

\subsection{Compactifying pseudo-Riemannian analogues of hyperbolic spaces and their Clifford--Klein forms} \label{subsec:compactif-Hpq}

The idea of embedding a group $G$ into a larger group $G'$ such that the homogeneous spaces $G/H$ can be realized explicitly as a $G$-orbit in an appropriate flag variety $G'/Q'$, can be applied in other cases as well.
For instance, for $p>q\geq 0$ and $\KK=\RR$, $\CC$, or~$\bH$, consider
\begin{align*}
G/H & = \Aut_{\KK}(b_{\KK}^{p,q+1})/\Aut_{\KK}(b_{\KK}^{p,q})\\
& \simeq\, \hat{\HH}_{\KK}^{p,q} : = \{ x\in\KK^{p+q+1} ~|~ b_{\KK}^{p,q+1}(x,x) = -1\},
\end{align*}
where we set
\begin{equation} \label{eqn:bpq}
b_{\KK}^{p,q}(x,x) := \overline{x}_1 x_1 + \dots + \overline{x}_p x_p - \overline{x}_{p+1} x_{p+1} - \dots - \overline{x}_{p+q} x_{p+q}.
\end{equation}
Explicitly, $\Aut_{\KK}(b_{\KK}^{p,q}) = \OO(p,q)$, $\U(p,q)$, or $\Sp(p,q)$, depending on whether $\KK=\RR$, $\CC$, or~$\bH$.
Note that $\hat{\HH}_{\KK}^{p,q}$ is homeomorphic to $\KK^{p}\times\mathbb{S}_{\KK}^q$, where \( \mathbb{S}_{\KK}^q = \{ z\in\KK^{q+1} \mid b_{\KK}^{q+1,0}(z,z)=1\}\) is itself homeomorphic to the real sphere $\mathbb{S}_{\RR}^{(q+1)\dim_{\RR}(\KK)-1}$.
The quotient $\HH_{\RR}^{p,q}$ of $\hat{\HH}_{\RR}^{p,q}$ by the involution $x\mapsto -x$ is a pseudo-Riemannian analogue of the usual real hyperbolic space $\HH_{\RR}^n$: it is pseudo-Riemannian of signature $(p,q)$ and has constant negative curvature.
Similarly, the quotient $\HH_{\CC}^{p,q}$ of $\hat{\HH}_{\CC}^{p,q}$ by $\U(1) = \{ z\in\CC ~|~ \overline{z}z=1\}$ and the quotient $\HH_{\bH}^{p,q}$ of $\hat{\HH}_{\bH}^{p,q}$ by $\Sp(1) = \{ z\in\bH ~|~ \overline{z}z=1\}$ are analogues of the usual complex and quaternionic hyperbolic spaces.
Let $P_{q+1}(b_{\KK}^{p,q+1})$ be the stabilizer in $\Aut_{\KK}(b_{\KK}^{p,q+1})$ of a maximal $b_{\KK}^{p, q+1}$-isotropic subspace of~$\KK^{p+q+1}$.
We prove the following.

\begin{theorem}\label{thm:Hpq}
Let $G = \Aut_{\KK}(b_{\KK}^{p,q+1})$ where $\KK=\RR$, $\CC$, or~$\bH$ and $p>q\geq 0$.
\begin{enumerate}
  \item \label{item:compactif-Hpq} The space $\F_1(b_{\KK}^{p+1,q+1})$ is a smooth compactification of $\hat{\HH}_{\KK}^{p,q}$, which we denote by $\overhatHpq$. 
It is the union of two $G$-orbits: an open one isomorphic to $\hat{\HH}_{\KK}^{p,q}$ and a closed one equal to $\F_1(b_{\KK}^{p,q+1})$, which we denote~by~$\partial \hat{\HH}_{\KK}^{p,q}$.
\end{enumerate}
Let $\Gamma$ be a word hyperbolic group and $\rho : \Gamma \to G$ a $P_{q+1}(b_{\KK}^{p,q+1})$-Anosov representation.
\begin{enumerate}
\setcounter{enumi}{1}
  \item \label{item:proper-Hpq} The action of $\Gamma$ on $\hat{\HH}_{\KK}^{p,q}$ via~$\rho$ is properly discontinuous, except possibly if $\KK=\RR$ and $p=q+1$.
  \item \label{item:compactif-quotient-Hpq} Assume that the action is properly discontinuous.
  Let $\xi: \partial_{\infty}\Gamma \to\nolinebreak \F_{q+1}(b_{\KK}^{p, q+1})$ be the boundary map of~$\rho$ and $\specialK_\xi$ the subset of $\partial \hat{\HH}_{\KK}^{p,q} = \F_1(b_{\KK}^{p,q+1})$ consisting of lines $\ell$ contained in $\xi(\eta)$ for some $\eta \in \partial_{\infty}\Gamma$.
Then $\Gamma$ acts properly discontinuously, via~$\rho$, on $\overhatHpq \smallsetminus \specialK_\xi$, and the quotient $\rho(\Gamma) \backslash (\overhatHpq \smallsetminus \specialK_\xi)$ is compact.
In particular, if $\Gamma$ is torsion-free, then $\rho(\Gamma) \backslash (\overhatHpq \smallsetminus \specialK_\xi)$ is a smooth compactification of $\rho(\Gamma) \backslash \hat{\HH}_{\KK}^{p,q}$.
\end{enumerate}
\end{theorem}

Suppose $\KK=\RR$ and $q=0$.
Then Theorem~\ref{thm:Hpq}.\eqref{item:compactif-Hpq} describes the usual compactification of the disjoint union of two copies of the real hyperbolic space~$\HH_{\RR}^p$, obtained by embedding them as two open hemispheres into the visual boundary $\partial\HH_{\RR}^{p+1} = \F_1(b_{\RR}^{p+1,1}) \simeq \mathbb{S}_{\RR}^p$ of $\HH_{\RR}^{p+1}$.
A representation $\rho : \Gamma\to\OO(p,1)$ is $P_1(b_{\RR}^{p,1})$-Anosov if and only if it is convex cocompact, in which case Theorem~\ref{thm:Hpq}.\eqref{item:proper-Hpq} states that $\Gamma$ acts properly discontinuously, via~$\rho$, on the complement in $\partial\HH_{\RR}^{p+1}$ of the limit set $\specialK_\xi$ of $\rho$ in $\partial\HH_{\RR}^p$.
When $\rho(\Gamma)\subset\SO(p,1)$, Theorem~\ref{thm:Hpq}.\eqref{item:compactif-quotient-Hpq} describes the compactification of two copies of the convex cocompact hyperbolic manifold $\rho(\Gamma)\backslash\HH_{\RR}^p$ obtained by gluing them along their common boundary.

For $\KK=\RR$ and $q=1$ (Lorentzian case), Theorem~\ref{thm:Hpq}.\eqref{item:compactif-Hpq} describes the usual compactification of the double cover of the anti-de Sitter space $\mathrm{AdS}^{p+1}$, obtained by embedding it into the Einstein universe $\mathrm{Ein}^{p+1}$.

In general, the compactification $\overhatHpq = \F_1(b_{\KK}^{p+1,q+1})$ is homeomorphic to 
\[ (\mathbb{S}_{\KK}^p\times\mathbb{S}_{\KK}^q)/\{ z\in\KK \mid \overline{z}z=1\} .\]

\begin{remark} \label{rem:O(p,p)}
For $\KK=\RR$ and $p=q+1$, the fact that $\rho$ is $P_{q+1}(b_{\KK}^{p,q+1})$-Anosov does not imply that the action of $\Gamma$ on $\hat{\HH}_{\KK}^{p,q}$ is properly discontinuous: see Example~\ref{ex:O(2,2)}.
In the case that $\KK=\RR$ and $p=q+1$ is odd, the action of $\Gamma$ on $\hat{\HH}_{\KK}^{p,q}$ can actually never be properly discontinuous unless $\Gamma$ is virtually cyclic, by \cite{Kassel_corank1}.
\end{remark}

\subsection{Compactifying more families of homogeneous spaces and their Clifford--Klein forms}

As a generalization of Theorem~\ref{thm:Hpq}, we prove the following.

\begin{proposition} \label{prop:compactification_general}
Let $(G,H,P,G',P')$ be as in Table~\ref{table2}.
\begin{enumerate}
  \item \label{item:compactif_general_G/H} There exists an open $G$-orbit $\mathcal{U}$ in $G'/P'$ that is diffeomorphic to $G/H$; the closure $\overline{\mathcal{U}}$ of $\mathcal{U}$ in $G'/P'$ provides a compactification of $G/H$.
  \item For any word hyperbolic group~$\Gamma$ and any $P$-Anosov representation $\rho: \Gamma \to G$, the cocompact domain of discontinuity $\Omega\subset G'/P'$ for $\rho(\Gamma)$ constructed in \cite{Guichard_Wienhard_DoD} (see Proposition~\ref{prop:dod_opq}) contains~$\mathcal{U}$; the quotient $\rho(\Gamma)\backslash (\Omega\cap\overline{\mathcal{U}})$ provides a compactification of $\rho(\Gamma)\backslash G/H$.
\end{enumerate}
\end{proposition}

\begin{table}[h!]
\centering
\begin{adjustwidth}{-1.2in}{-1.2in}
\centering
\begin{tabular}{|c|c|c|c|c|c|}
\hline
& $G$ & $H$ & $P$ & $G'$ & $P'$\tabularnewline
\hline
(i) & $\OO(p,q+1)$ & $\OO(p,q)$ & $\mathrm{Stab}_G(W)$ & $\OO(p+1,q+1)$ & $\mathrm{Stab}_{G'}(\ell')$ \tabularnewline
(ii) & $\U(p,q+1)$ & $\U(p,q)$ & $\mathrm{Stab}_G(W)$ & $\U(p+1,q+1)$ & $\mathrm{Stab}_{G'}(\ell')$\tabularnewline
(iii) & $\Sp(p,q+1)$ & $\Sp(p,q)$ & $\mathrm{Stab}_G(W)$ & $\Sp(p+1,q+1)$ & $\mathrm{Stab}_{G'}(\ell')$\tabularnewline
(iv) & $\OO(2p,2q)$ & $\U(p,q)$ & $\mathrm{Stab}_G(\ell)$ & $\OO(2p+2q,\CC)$ & $ \mathrm{Stab}_{G'}(W')$\tabularnewline
(v) & $\U(2p,2q)$ & $\Sp(p,q)$ & $\mathrm{Stab}_G(\ell)$ & $\Sp(p+q,p+q)$ &$\mathrm{Stab}_{G'}(W')$\tabularnewline
(vi) & $\Sp(2m,\RR)$ & $\U(p,m-p)$ & $\mathrm{Stab}_G(\ell)$ & $\Sp(2m,\CC)$ & $ \mathrm{Stab}_{G'}(W')$\tabularnewline
\hline
\end{tabular}
\end{adjustwidth}
\vspace{0.2cm}
\caption{Reductive groups $H\subset G\subset G'$ and parabolic subgroups $P$ of~$G$ and $P'$ of~$G'$ to which Proposition~\ref{prop:compactification_general} applies.
Here $m,p,q$ are any integers with $m>0$; we require $p>q+1$ in case~(i) and $p>q$ in cases (ii), (iii), as well as $q>0$ in cases (iv), (v).
We denote by $\ell$ or~$\ell'$ an isotropic line and by $W$ or~$W'$ a maximal isotropic subspace (over $\RR$, $\CC$, or~$\bH$), relative to the form $b$ preserved by $G$ or~$G'$.}
\label{table2}
\end{table}

In all examples of Table~\ref{table2}, the space $G/H$ is an affine symmetric space.
Examples (i), (ii), (iii) correspond to the situation of Section~\ref{subsec:compactif-Hpq}.
In examples (i), (ii), (iii) with $q=0$ and example~(vi) with~$p=0$, the symmetric space $G/H$ is Riemannian. 
Example~(i), example~(iv) with $q =1$, and example~(vi) with $p=0$ were previously described in \cite{Guichard_Wienhard_DoD}.
The open $G$-orbit $\mathcal{U}$ diffeomorphic to $G/H$ is given explicitly in each case in Section~\ref{sec:proof-comp-general}.

\begin{remark} \label{rem:lift-DoD}
The cocompact domains of discontinuity we describe in $G'/P'$ lift to cocompact domains of discontinuity in $G'/P''$ for any parabolic subgroup $P''$ of~$G'$ contained in~$P'$; in particular, they lift to cocompact domains of discontinuity in $G'/P'_{\mathrm{min}}$ where $P'_{\mathrm{min}}$ is a minimal parabolic subgroup of~$G'$.
The compactifications we describe for quotients $\rho(\Gamma)\backslash G/H$ induce compactifications of the quotients $\rho(\Gamma)\backslash\mathcal{U}'$ for any $G$-orbit $\mathcal{U}'$ in $G'/P''$ lifting the \(G\)-orbit \(\mathcal{U} \subset G'/P'\) diffeomorphic~to~$G/H$. 
\end{remark}

In Proposition~\ref{prop:compactification_general_nonsym} below, we also treat homogeneous spaces that are not affine symmetric spaces using Remark~\ref{rem:lift-DoD}.

\subsection{Tameness}

We establish the topological tameness of the Clifford--Klein forms $\rho(\Gamma)\backslash G/H$ above.
Recall that a manifold is said to be \emph{topologically tame} if it is homeomorphic to the interior of a compact manifold with boundary.

In the setting of Theorem~\ref{thm:Hpq}, both $\rho(\Gamma) \backslash (\hat{\HH}_{\KK}^{p,q} \cup \mathcal{C}_\xi)$ and $\rho(\Gamma) \backslash \mathcal{C}_\xi$ are smooth compact manifolds, without boundary.
The complement, inside a compact manifold, of a compact submanifold is topologically tame.
Therefore, Theorem~\ref{thm:Hpq} immediately yields the following.

\begin{corollary}\label{cor:tame-Hpq}
Let $\Gamma$ be a torsion-free word hyperbolic group, $p>q\geq 0$ two integers.
For any $P_{q+1}(b_{\KK}^{p,q+1})$-Anosov representation $\rho : \Gamma \to \Aut_{\KK}(b_{\KK}^{p,q+1})$, if the action of $\Gamma$ on \(\hat{\HH}_{\KK}^{p,q}\) via~$\rho$ is properly discontinuous (see Theorem~\ref{thm:Hpq}.\eqref{item:proper-Hpq}), then the manifold $\rho(\Gamma) \backslash \hat{\HH}_{\KK}^{p,q}$ is topologically tame.
\end{corollary}

In order to prove topological tameness in more general cases, we establish the following useful result.

\begin{lemma} \label{lem:tameness_intro}
Let $G\subset G'$ be two real reductive algebraic groups and $\Gamma$ a torsion-free discrete subgroup of~$G$. 
Let $X$ be a $G'$-homogeneous space and $\Omega$ an open subset of~$X$ on which $\Gamma$ acts properly discontinuously and cocompactly.
For any $G$-orbit $\mathcal{U}\subset \Omega$, the quotient $\Gamma \backslash \mathcal{U}$ is a topologically tame manifold.
\end{lemma}

Proposition~\ref{prop:compactification_general} and
Lemma~\ref{lem:tameness_intro} immediately imply the following, by taking $\mathcal{U}$ to be a $G$-orbit in $G'/Q'$ that identifies with $G/H$.

\begin{corollary}\label{cor:tame-table2}
Let $\Gamma$ be a torsion-free word hyperbolic group and let $H\subset\nolinebreak G\supset\nolinebreak P$ be as in Table~\ref{table2}.
For any $P$-Anosov representation $\rho: \Gamma \to G$, the quotient $\rho(\Gamma)\backslash G/H$ is a topologically tame manifold.
\end{corollary}

Using Theorem~\ref{thm:intro_comp_quotient} and Lemma~\ref{lem:tameness_intro}, we also prove the following.

\begin{theorem}\label{thm:tame-G/Gamma}
Let $\Gamma$ be a torsion-free word hyperbolic group, $G$ a real reductive Lie group, and $P$ a proper parabolic subgroup of~$G$.
For any $P$-Anosov representation $\rho : \Gamma\to G$, the quotient $\rho(\Gamma)\backslash G$ a is topologically tame manifold.
\end{theorem}

\begin{remark}
In the Riemannian case, compactifications of quotients of symmetric spaces have recently been constructed by a different method in \cite{KapovichLeeb15} for uniformly $\sigma_{mod}$-regular and conical discrete subgroups of~$G$; this class of discrete groups contains the images of $P_{\mathrm{min}}$-Anosov representations where $P_{\mathrm{min}}$ is a minimal parabolic subgroup of~$G$.
\end{remark}

\subsection{Organization of the paper}

In Section~\ref{sec:compactif-group} we establish Theorem~\ref{thm:compactif-Aut(b)} and its analogue for $\GL_{\KK}(V)$ (Theorem~\ref{thm:compactif-GL}).
In Section~\ref{sec:rem-Anosov} we recall the notion of Anosov representation, the construction of domains of discontinuity from \cite{Guichard_Wienhard_DoD}, and a few facts from \cite{GGKW_anosov} on Anosov representations into $\Aut_{\KK}(b)\times\Aut_{\KK}(b)$.
This allows us, in Section~\ref{sec:dod-GxG}, to establish Theorem~\ref{thm:intro_comp_quotient} and some generalization (Theorem~\ref{thm:compactification}).
In Section~\ref{sec:proof-comp-general} we prove Theorem~\ref{thm:Hpq} and Proposition~\ref{prop:compactification_general}.
Finally, Section~\ref{sec:tameness} is devoted to topological tameness, with a proof of Lemma~\ref{lem:tameness_intro} and Theorem~\ref{thm:tame-G/Gamma}.

\subsection*{Acknowledgements}

We are grateful to Jeff Danciger and Pablo Solis for useful comments and discussions.

\section{Compactification of group manifolds} \label{sec:compactif-group}

In this section we provide a short proof of Theorem~\ref{thm:compactif-Aut(b)} and of its analogue for general linear groups $\GL_{\KK}(V)$ with $\KK=\RR$, $\CC$, or~$\bH$ (Theorem~\ref{thm:compactif-GL}).

\subsection{The case $G = \Aut_{\KK}(b)$}

Let us prove Theorem~\ref{thm:compactif-Aut(b)}.
We use the notation of Section~\ref{subsec:intro-compact-G}.
In particular,
\[ \pi : \F_N(b\oplus -b) \longrightarrow \bigg(\bigcup_{i=0}^n \F_i(b)\bigg) \times \bigg(\bigcup_{i=0}^n \F_i(-b)\bigg) \]
is the map defined by \eqref{eqn:pi}.
The group
\[ \Aut_{\KK}(b) \times \Aut_{\KK}(b) = \Aut_{\KK}(b) \times \Aut_{\KK}(-b) \]
naturally embeds into $\Aut_{\KK}(b \oplus -b)$.
For $0\leq i\leq n$, the set
\[ \mathcal{U}_i := \pi^{-1}\big(\F_i(b) \times \F_i(-b)\big) \]
is clearly invariant under $\Aut_{\KK}(b) \times \Aut_{\KK}(-b)$.

\begin{lemma}\label{lem:sets-equal}
The space $X = \F_N(b\oplus -b)$ of maximal $(b\oplus -b)$-isotropic $\KK$-subspaces of $V \oplus V$ is the union of the sets~$\mathcal{U}_i$ for $0\leq i\leq n$.
\end{lemma}

\begin{proof}
It is enough to prove that for any $W\in\F_N(b\oplus -b)$,
\begin{equation} \label{eqn:same-dim}
\dim_{\KK} (W \cap (\{ 0\} \oplus V)) = \dim_{\KK}(W \cap (V\oplus \{0\})).
\end{equation}
We have
\begin{align*}
\dim_{\KK} (W \cap (\{ 0\} \oplus V)) & =  \dim_{\KK}(W) + \dim_{\KK}(\{0\}\oplus V) - \dim_{\KK}(W + (\{ 0\} \oplus V))\\
& =  \dim_{\KK}(V\oplus V) - \dim_{\KK}(W + (\{ 0\} \oplus V))\\
& =  \dim_{\KK}(W + (\{ 0\} \oplus V))^{\perp},
\end{align*}
where $(W + (\{ 0\} \oplus V))^{\perp}$ denotes the orthogonal complement of $W + (\{ 0\} \oplus V)$ in $V\oplus V$ with respect to $b\oplus -b$.
But
\[ (W+(\{ 0\} \oplus V))^{\perp} = W^{\perp} \cap (\{ 0\} \oplus V) ^{\perp} = W^{\perp} \cap (V\oplus \{0\}), \]
hence $\dim_{\KK}(W \cap (\{ 0\} \oplus V)) = \dim_{\KK}(W^{\perp} \cap (V\oplus \{0\}))$.
Since $W$ is maximal isotropic for $b\oplus -b$, we have $W = W^{\perp}$, and so \eqref{eqn:same-dim} holds.
\end{proof}

For any $0\leq i\leq n$, let
\[\pi_i : \mathcal{U}_i \longrightarrow \F_i(b) \times \F_i(-b) \]
be the projection induced by~$\pi$.
By construction, $\pi_i$ is $(\Aut_{\KK}(b) \times \Aut_{\KK}(b))$-equivariant, hence surjective (because the action of $\Aut_{\KK}(b)$ on $\F_i(b) = \F_i(-b)$ is transitive).
We now describe the fibers of~$\pi_i$.
By equivariance and surjectivity, it is enough to determine the fiber of~\(\pi_i\) above one particular point of \(\F_i(b) \times\nolinebreak\F_i(b)\).

For any $V_i\in\F_i(b)$, let $b_{V_i}$ be the $\RR$-bilinear form induced by $b$ on $V_i^{\perp_b}/V_i\simeq\KK^{\dim_{\KK}(V)-2i}$.
If $b$ is symmetric, antisymmetric, Hermitian, or anti-Hermitian, then so is~$b_{V_i}$.
For instance, if $b$ is symmetric over~$\RR$ with signature $(p,q)$, then $b_{V_i}$ has signature $(p-i,q-i)$.

\begin{lemma} \label{lem:U=GxG/Diag}
For any $V_i\in\F_i(b)$, the fiber $\pi_i^{-1}(V_i,V_i) \subset \F_N(b\oplus -b)$ is the set of maximal $(b\oplus -b)$-isotropic $\KK$-subspaces of $V_i^{\perp_b}\oplus V_i^{\perp_b}$ that contain $V_i\oplus V_i$ and project to maximal isotropic subspaces of $(V_{i}^{\perp_b}/V_i) \oplus (V_{i}^{\perp_b}/V_i)$ transverse to both factors $(V_{i}^{\perp_b}/V_i) \oplus \{0\}$ and $\{0\} \oplus (V_{i}^{\perp_b}/V_i)$.
As an $(\Aut_{\KK}(b_{V_i}) \times \Aut_{\KK}(b_{V_i}))$-space, $\pi_i^{-1}(V_i,V_i)$ is isomorphic to
\[ (\Aut_{\KK}(b_{V_i}) \times \Aut_{\KK}(b_{V_i}))/\Diag(\Aut_{\KK}(b_{V_i})). \]
\end{lemma}

\begin{proof}
By definition, any $W\in\pi_i^{-1}(V_i,V_i)$ satisfies $W \cap (V \oplus \{ 0\}) = V_i \oplus\nolinebreak \{ 0\}$ and $W \cap (\{ 0\} \oplus V) = \{ 0\} \oplus V_i$, hence $W$ contains $V_i\oplus V_i$ and $W \subset V_i^{\perp_b} \oplus V_i^{\perp_b}$ since $W$ is $(b\oplus\nolinebreak -b)$-isotropic.
Thus $\pi_i^{-1}(V_i,V_i)$ is the set of maximal $(b\oplus\nolinebreak -b)$-isotropic subspaces of $V_i^{\perp_b}\oplus V_i^{\perp_b}$ that contain $V_i\oplus V_i$ and correspond to maximal isotropic subspaces of $(V_{i}^{\perp_b}/V_i) \oplus (V_{i}^{\perp_b}/V_i)$ transverse to both factors.
In particular, $\pi_i^{-1}(V_i,V_i)$ identifies with its image in $\F_{N-2i}(b_{V_i}\oplus\nolinebreak -b_{V_i})$ and is endowed with an action of $\Aut_{\KK}(b_{V_i}) \times \Aut_{\KK}(b_{V_i})$.

We first check that this action of $\Aut_{\KK}(b_{V_i}) \times \Aut_{\KK}(b_{V_i})$ is transitive.
Let $W'_0$ be the image in $(V_i^{\perp_b}/V_i) \oplus (V_i^{\perp_b}/V_i)$ of
\[ \{ (v,v) \,|\, v\in V_i^{\perp_b}\} \subset V_i^{\perp_b}\oplus V_i^{\perp_b}. \]
The image $W'$ in $(V_i^{\perp_b}/V_i) \oplus (V_i^{\perp_b}/V_i)$ of any element of $\pi_i^{-1}(V_i,V_i)$ meets the second factor $V_i^{\perp_b}/V_i$ trivially, hence is the graph of some linear endomorphism $h$ of $V_i^{\perp_b}/V_i$.
This $h$ belongs to $\Aut_{\KK}(b_{V_i})$ since $W'$ is $(b_{V_i}\oplus -b_{V_i})$-isotropic.
Thus $W' = (e,h) \cdot W'_0$ lies in the $(\Aut_{\KK}(b_{V_i}) \times \Aut_{\KK}(b_{V_i}))$-orbit of~$W'_0$, proving transitivity.

Let us check that the stabilizer of $W'_0$ in $\Aut_{\KK}(b_{V_i}) \times \Aut_{\KK}(b_{V_i})$ is the diagonal $\Diag(\Aut_{\KK}(b_{V_i}))$.
For any $(g_1,g_2) \in \Aut_{\KK}(b_{V_i}) \times \Aut_{\KK}(b_{V_i})$,
\[ (g_1, g_2) \cdot W'_0 = \{ (g_1(v), g_2(v)) ~|~ v \in V_i^{\perp_b}/V_i \} = \{ (v, g_1^{-1}g_2(v)) ~|~ v \in V_i^{\perp_b}/V_i \}, \]
and so $(g_1,g_2) \cdot W'_0 = W'_0$ if and only if $g_1 = g_2$.
\end{proof}

In particular, taking $i=0$, we obtain that $\mathcal{U}_0$ is an $(\Aut_{\KK}(b) \times \Aut_{\KK}(b))$-space isomorphic to
\[ (\Aut_{\KK}(b) \times \Aut_{\KK}(b))/\Diag(\Aut_{\KK}(b)). \]

\begin{lemma}\label{lem:G-orbits}
For any $0\leq i\leq n$, the action of $\Aut_{\KK}(b) \times \Aut_{\KK}(b)$ on~$\mathcal{U}_i$ is transitive.
\end{lemma}

\begin{proof}
The map $\pi_i$ is $(\Aut_{\KK}(b)\times\Aut_{\KK}(b))$-equivariant and the action of\linebreak $\Aut_{\KK}(b)\times\Aut_{\KK}(b)$ on $\F_i(b) \times \F_i(-b)$ is transitive, hence it is enough to see that for any $V_i\in\F_i(b)$ the action of the stabilizer of $(V_i,V_i)$ in $\Aut_{\KK}(b)\times\Aut_{\KK}(b)$ is transitive on the fiber $\pi_i^{-1}(V_i,V_i)$.
This follows from Lemma~\ref{lem:U=GxG/Diag}.
\end{proof}

\begin{lemma}\label{lem:dimUi}
  For any \(0\leq i\leq n\), the dimension of \(\mathcal{U}_i\) is
  \begin{equation*}
    \dim_\RR (\mathcal{U}_i) = \dim_\RR(\Aut_{\KK}(b)) - i^2 \dim_{\RR}(\KK).
  \end{equation*}
\end{lemma}

\begin{proof}
Consider two elements $V_i,V'_i\in\mathcal{F}_i(b)$ such that \(V_{i}^{\perp_b} \cap V'_i=\{0\}\).
Let \(V_{2i}\) be the sum in~\(V\) of \(V_i\) and \(V'_i\), and \(W = V_{i}^{\perp_b} \cap V_{i}^{\prime\perp_b} \simeq V_{i}^{\perp_b}/V_i\).
The restrictions of \(b\) to \(V_{2i}\) and to \(W\) are nondegenerate and \(V\) is the direct \(b\)-orthogonal sum of
\(V_{2i}\) and~\(W\).
The parabolic subgroups \(P_i = \mathrm{Stab}_{\Aut_{\KK}(b)}(V_i)\) and \(P'_i = \mathrm{Stab}_{\Aut_{\KK}(b)}(V'_i)\) are conjugate in~\(G\) and the set 
\(P_i P'_i = \{ p_i p'_i \mid p_i\in P_i,\ p'_i\in P'_i\}\) is open in \(G\) since the sum of the Lie algebras of \(P_i\) and of~\(P'_i\) is equal to the Lie algebra of $\Aut_{\KK}(b)$.
Therefore,
\begin{equation*}
  \dim_\RR(\Aut_{\KK}(b)) = \dim_\RR(P_i P'_i) = 2 \dim_\RR(P_i) -\dim_\RR(P_i \cap P'_i).
\end{equation*}
It is easy to see that
\begin{eqnarray*}
P_i \cap P'_i & = & \big(\mathrm{Stab}_{\Aut_\KK(b|_{V_{2i}})}(V_i) \cap \mathrm{Stab}_{\Aut_\KK(b|_{V_{2i}})}(V'_i)\big) \times \Aut_\KK(b|_W)\\
& \simeq & \GL_{\KK}(V_i) \times \Aut_\KK(b_{V_i}).
\end{eqnarray*}
This implies
\begin{eqnarray*}
2 \dim_\RR(P_i) & = & \dim_\RR(\Aut_{\KK}(b)) - \dim_\RR(\GL_\KK(V_i)) - \dim_\RR(\Aut_\KK(b_{V_i}))\\
& = & \dim_\RR(\Aut_{\KK}(b)) - i^2 \dim_\RR(\KK) - \dim_\RR(\Aut_\KK(b_{V_i})).
\end{eqnarray*}
Using Lemma~\ref{lem:U=GxG/Diag}, we obtain
\begin{eqnarray*}
\dim_{\RR}(\mathcal{U}_i) & = & 2\dim_{\RR}(\F_i(b)) + \dim_{\RR}(\Aut_{\KK}(b_{V_i}))\\
& = & 2 \dim_\RR(\Aut_{\KK}(b)) - 2 \dim_\RR(P_i) + \dim_{\RR}(\Aut_{\KK}(b_{V_i}))\\
& = & \dim_\RR(\Aut_{\KK}(b)) - i^2 \dim_\RR(\KK). \qedhere
\end{eqnarray*}
\end{proof}

By Lemma~\ref{lem:dimUi}, we have $\dim_{\RR}(\mathcal{U}_i)>\dim_{\RR}(\mathcal{U}_j)$ for all $0\leq i<j\leq n$.
The function $W\mapsto\dim_{\RR}(W\cap (V\oplus\{ 0\}))$ is upper semicontinuous on $\F_N(b\oplus\nolinebreak -b)$.
Therefore, for any $0\leq i\leq n$, the closure $S_i$ of $\mathcal{U}_i$ in $\F_N(b\oplus -b)$ is the union of the submanifolds $\mathcal{U}_j$ for $i\leq j\leq n$.

By the Iwasawa decomposition, any maximal compact subgroup of $\Aut_{\KK}(b\oplus\nolinebreak -b)$ acts transitively on the flag variety $\F_N(b\oplus -b)$. 
By computing the stabilizer of a point in each case, we see that $\F_N(b\oplus -b)$ identifies with a Riemannian symmetric space of the compact type as in Table~\ref{table1}.
This completes the proof of Theorem~\ref{thm:compactif-Aut(b)}.

\subsection{The case $G = \GL_{\KK}(V)$}

We now establish an analogue of Theorem~\ref{thm:compactif-Aut(b)} when $G = \GL_{\KK}(V)$ is the full group of invertible $\KK$-linear transformations of~$V$.
Here we use the notation $\F_i(V)$ to denote the Grassmannian of $i$-dimensional $\KK$-subspaces of~$V$, and $N$ to denote $\dim_{\KK}(V)$.
Then \eqref{eqn:pi} defines a map
\[ \pi : \F_N(V\oplus V) \longrightarrow \bigg(\bigcup_{i=0}^N \F_i(V)\bigg) \times \bigg(\bigcup_{i=0}^N \F_i(V)\bigg). \]

\begin{theorem} \label{thm:compactif-GL}
Let $V$ be an $N$-dimensional vector space over $\KK=\RR$, $\CC$, or~$\bH$, and $G = \GL_{\KK}(V)$. 
The Grassmannian $X= \F_N(V\oplus V)$ of $N$-dimensional $\KK$-subspaces of $V\oplus V$ is a smooth compactification of the group manifold $(G\times G)/\Diag(G)$ with the following properties: 
\begin{enumerate}
  \item \label{item:cpt-sym-space-GL} $X$ is a real analytic manifold (in fact complex analytic if $\KK=\CC$).
  Under the action of a maximal compact subgroup of $\GL_{\KK}(V\oplus V)$, it identifies with a Riemannian symmetric space of the compact type, namely
   \begin{itemize}
    \item $\OO(2N)/(\OO(N)\times\OO(N))$ if $\KK=\RR$,
    \item $\U(2N)/(\U(N)\times\U(N))$ if $\KK=\CC$,
    \item $\Sp(2N)/(\Sp(N)\times\Sp(N))$ if $\KK=\bH$.
  \end{itemize}
  \item The $(G \times G)$-orbits in~$X$ are the submanifolds $\mathcal{U}_{i,j}  := \pi^{-1}(\F_i(V) \times \F_j(V))$ for $0\leq i+ j\leq N$; there are $(N+1)(N+2)/2$ of them.
  They have dimension $\dim_{\KK}(\mathcal{U}_{i,j}) = \dim_{\KK}(G) - i^2 - j^2$.
  The closure of $\mathcal{U}_{i,j}$ in $X$ is $\bigcup_{k\geq i,\, \ell\geq j} \mathcal{U}_{k,\ell}$.
  \item For $0\leq i+j\leq N$, the map $\pi$ defines a fibration $\pi_{i,j}$ of $\mathcal{U}_{i,j}$ over $\F_i(V) \times \F_j(V)$ with fibers given by Lemma~\ref{lem:U=GxG/Diag-GL} below.
\end{enumerate}
In particular, \(\mathcal{U}_{0,0}\) is the unique open \((G\times G)\)-orbit in~$X$ and it identifies with\linebreak \((G\times G) / \Diag(G)\).
\end{theorem}

The proof of Theorem~\ref{thm:compactif-GL} is similar to Theorem~\ref{thm:compactif-Aut(b)}: the group $\GL_{\KK}(V) \times \GL_{\KK}(V)$ naturally embeds into $\GL_{\KK}(V \oplus V)$.
For $i,j\in\NN$ with $i+j\leq N$, the set
\[ \mathcal{U}_{i,j} := \pi^{-1}\big(\F_i(V) \times \F_j(V)\big) \]
is clearly invariant under $\GL_{\KK}(V) \times \GL_{\KK}(V)$, and $X = \F_N(V\oplus V)$ is the union of these sets~$\mathcal{U}_{i,j}$.
Let
\[\pi_{i,j} : \mathcal{U}_{i,j} \longrightarrow \F_i(V) \times \F_j(V) \]
be the projection induced by~$\pi$.
By construction, $\pi_{i,j}$ is $(\GL_{\KK}(V) \times \GL_{\KK}(V))$-equivariant, hence surjective (because the action of $\GL_{\KK}(V)$ on $\F_i(V)$ and $\F_j(V)$ is transitive).
As above, it is enough to determine the fiber of~\(\pi_{i,j}\) above one particular point of \(\F_i(V) \times \F_j(V)\).
Let $(e_1,\dots,e_N)$ be a basis of~$V$.
We set
\[ \left\{ \begin{array}{ccl}
V_i & := & \mathrm{span}_{\KK}(e_1,\dots,e_i),\\
V'_i & := & \mathrm{span}_{\KK}(e_{i+1},\dots,e_N),\\
V_j & := & \mathrm{span}_{\KK}(e_{N-j+1},\dots,e_N),\\
V'_j & := & \mathrm{span}_{\KK}(e_1,\dots,e_{N-j}),\\
V'_{i,j} & := & V'_i\cap V'_j = \mathrm{span}_{\KK}(e_{i+1},\dots,e_{N-j}),
\end{array} \right.\]
so that $V$ is the direct sum of $V_i$ and~$V'_i$, and also of $V_j$ and~$V'_j$.
By assumption, $i+j\leq N$, hence $V_i\cap V_j=\{ 0\}$.
The quotient $V/V_i$ identifies with $V'_i$, which is the direct sum of $V'_{i,j}:=V'_i\cap V'_j$ and~$V_j$.
Similarly, the quotient $V/V_j$ identifies with $V'_j$, which is the direct sum of $V'_{i,j}$ and~$V_i$.
We see $(V_i,V_j)$ as an element of $\F_i(V)\times\F_j(V)$.

\begin{lemma} \label{lem:U=GxG/Diag-GL}
The fiber $\pi_{i,j}^{-1}(V_i,V_j) \subset \F_N(V\oplus V)$ is the set of $N$-dimensional $\KK$-subspaces of $V\oplus V$ that contain $V_i\oplus V_j$ and project to $(N-i-j)$-dimen\-sional $\KK$-subspaces of $(V/V_i) \oplus (V/V_j)$ transverse to both factors $(V/V_i) \oplus\nolinebreak\{0\}$ and $\{0\} \oplus (V/V_j)$.
As a $\GL_{\KK}(V/V_i) \times \GL_{\KK}(V/V_j)$-space, $\pi_i^{-1}(V_i,V_j)$ is isomorphic to the quotient of $\GL_{\KK}(V/V_i) \times \GL_{\KK}(V/V_j) \simeq \GL_\KK(V'_i) \times \GL_\KK(V'_j)$ by the subgroup
\begin{equation} \label{eqn:stab-GL-V'i-V'j}
\bigl(\GL_\KK(V_j) \times \GL_\KK(V_i) \times \Diag(\GL_\KK(V'_{i,j}))\bigr) \ltimes \bigl( (V_i^* \otimes V'_{i,j}) \oplus (V_j^* \otimes V'_{i,j})\bigr).
\end{equation}
\end{lemma}

\begin{proof}
The first statement is clear.
For the second statement, one easily checks that $\pi_{i,j}^{-1}(V_i,V_j)$ is the $(\GL_{\KK}(V'_i) \times \GL_{\KK}(V'_j))$-orbit of
\[ W_0 := (V_i \oplus \{ 0\}) + (\{ 0\}\oplus V_j) + \big\{ (v,v) ~|~ v\in V'_{i,j}\big\} \]
and that the stabilizer of $W_0$ in $\GL_{\KK}(V'_i) \times \GL_{\KK}(V'_j)$ is \eqref{eqn:stab-GL-V'i-V'j}.
\end{proof}

In particular, $\mathcal{U}_{0,0}$ is a $(\GL_{\KK}(V) \times \GL_{\KK}(V))$-space isomorphic to
\[ (\GL_{\KK}(V) \times \GL_{\KK}(V))/\Diag(\GL_{\KK}(V)). \]
Similarly to Lemma~\ref{lem:G-orbits}, for any $i,j\in\NN$ with $i+j\leq N$, the action of $\GL_{\KK}(V) \times \GL_{\KK}(V)$ on~$\mathcal{U}_{i,j}$ is transitive.
Note that $\dim_{\KK}(\F_i(V)) = i(N-i)$.
From Lemma~\ref{lem:U=GxG/Diag-GL} we compute 
$\dim_\KK(\pi_{i,j}^{-1}(V_i,V_j)) = N^2 - (i+j)N$, and so
\begin{eqnarray*}
\dim_{\KK}(\mathcal{U}_i) & = & \dim_{\KK}(\F_i(V)) + \dim_{\KK}(\F_j(V)) + \dim_\KK(\pi_{i,j}^{-1}(V_i,V_j))\\
& = & N^2 - i^2 - j^2. 
\end{eqnarray*}
In particular, $\dim_{\RR}(\mathcal{U}_{i,j})>\dim_{\RR}(\mathcal{U}_{k,\ell})$ for all $(i,j)\neq (k,\ell)$ with $i\leq k$ and $j\leq\ell$.
By uppersemicontinuity of the functions $W\mapsto\dim_{\RR}(W\cap (V\oplus\{ 0\}))$ and $W\mapsto\dim_{\RR}(W\cap (\{ 0\}\oplus V))$, the closure $S_{i,j}$ of $\mathcal{U}_{i,j}$ in $\F_N(V\oplus V)$ is the union of the submanifolds $\mathcal{U}_{k,\ell}$ for $k\geq i$ and $\ell\geq j$.

By the Iwasawa decomposition, any maximal compact subgroup of $\GL_{\KK}(V\oplus V)$ acts transitively on the flag variety $\F_N(V\oplus V)$.
By computing the stabilizer of a point, we see that $\F_N(V\oplus V)$ identifies with a Riemannian symmetric space of the compact type as in Theorem~\ref{thm:compactif-GL}.\eqref{item:cpt-sym-space-GL}.
This completes the proof of Theorem~\ref{thm:compactif-GL}.

\section{Reminders on Anosov representations and their\\ domains of discontinuity} \label{sec:rem-Anosov}

In this section we recall the definition of an Anosov representation into a reductive Lie group, see \cite{Labourie_anosov, Guichard_Wienhard_DoD, GGKW_anosov}, and the construction of domains of discontinuity given in \cite{Guichard_Wienhard_DoD}.
We first introduce some notation.

\subsection{Notation}

Let $G$ be a real reductive Lie group with Lie algebra~$\g$.
We assume $G$ to be noncompact, equal to a finite union of connected components (for the real topology) of $\mathbf{G}(\RR)$ for some algebraic group~$\mathbf{G}$.
Then $\g = \mathfrak{z}(\g) + \g_s$, where $\mathfrak{z}(\g)$ is the Lie algebra of the center of~$G$ and $\g_s$ the Lie algebra of the derived subgroup of~$G$, which is semisimple.
Let $K$ be a maximal compact subgroup of~$G$, with Lie algebra~$\mathfrak{k}$, and let $\aaa = (\aaa\cap\mathfrak{z}(\g)) + (\aaa\cap\g_s)$ be a maximal abelian subspace of the orthogonal complement of $\mathfrak{k}$ in~$\g$ for the Killing form (in Section~\ref{sec:proof-comp-general} we shall call this a \emph{Cartan subspace} of~$\g$).
The \emph{real rank} of~$G$ is by definition the dimension of~$\aaa$.
Let $\Sigma$ be the set of restricted roots of $\aaa$ in~$\g$, \ie the set of nonzero linear forms $\alpha\in\aaa^*$ for which
\[ \g_{\alpha} := \{ z\in\g ~|~ \ad(a)(z) = \alpha(a)\,z \quad\forall a\in\aaa\} \]
is nonzero.
Let $\Delta\subset\Sigma$ be a choice of system of \emph{simple roots}, \ie any element of~$\Sigma$ is expressed uniquely as a linear combination of elements of \(\Delta\) with coefficients all of the same sign.
Let 
\[ \overline{\aaa}^+ := \{Y \in \aaa \mid \alpha(Y) \geq 0\quad \forall \alpha \in \Delta\} \]
be the closed positive Weyl chamber of~$\aaa$ associated with~$\Delta$.
The \emph{Weyl group} of $\aaa$ in~$\g$ is the group $W=N_K(\aaa)/Z_K(\aaa)$, where $N_K(\aaa)$ (\resp $Z_K(\aaa)$) is the normalizer (\resp centralizer) of $\aaa$ in~$K$.
There is a unique element $w_0 \in W$ such that $w_0\cdot\overline{\aaa}^+=-\overline{\aaa}^+$; the involution of~$\aaa$ defined by $Y\mapsto -w_0\cdot Y$ is called the \emph{opposition involution}.
The corresponding dual linear map preserves~$\Delta$; we shall denote it by
\begin{align}\label{eqn:opp-inv}
  \aaa^{\ast} & \longrightarrow  \aaa^{\ast}\\
  \alpha\, & \longmapsto  \alpha^{\star} = -w_0\cdot \alpha. \notag
\end{align}
Recall that the \emph{Cartan decomposition} $G=K(\exp\overline{\aaa}^+)K$ holds: any $g\in G$ may be written $g=k(\exp\mu(g))k'$ for some $k,k'\in K$ and a unique $\mu(g)\in\overline{\aaa}^+$ (see \cite[Ch.\,IX, Th.\,1.1]{Helgason}).
This defines a map
\[ \mu : G \longrightarrow \overline{\aaa}^+ \]
called the \emph{Cartan projection}, inducing a homeomorphism $K \backslash G / K \simeq \overline{\aaa}^+$.
We refer to \cite[\S\,2]{GGKW_anosov} for more details.

Let $\Sigma^+ \subset \Sigma$ be the set of positive roots with respect to~$\Delta$, \ie roots that are nonnegative linear combinations of elements of~\(\Delta\).
For any nonempty subset $\theta\subset\Delta$, we denote by $P_\theta$ the normalizer in~$G$ of the Lie algebra $\mathfrak{u}_\theta = \bigoplus_{\alpha \in \Sigma_{\theta}^{+}} \g_\alpha$, where $\Sigma_{\theta}^{+} = \Sigma^+ \smallsetminus  \mathrm{span}(\Delta \smallsetminus \theta)$ is the set of positive roots that do \emph{not} belong to the span of $\Delta \smallsetminus \theta$.
Explicitly,
\[\mathrm{Lie}(P_{\theta}) = \g_0 \oplus \bigoplus_{\alpha\in\Sigma^+} \g_{\alpha} \oplus \bigoplus_{\alpha\in\Sigma^+\cap\mathrm{span}(\Delta\smallsetminus\theta)} \g_{-\alpha}.\]
In particular, $P_\emptyset = G$ and $P_\Delta$ is a minimal parabolic subgroup of~$G$.\footnote{This is the same convention as in \cite{GGKW_anosov}, but the opposite convention to \cite{Guichard_Wienhard_DoD}.}
Any parabolic subgroup of~$G$ is conjugate to~$P_{\theta}$ for some $\theta\subset\Delta$.

\subsection{Anosov representations} \label{subsec:def-Ano}

The following definition of Anosov representations is not the original one from \cite{Labourie_anosov, Guichard_Wienhard_DoD}, but an equivalent one taken from \cite{GGKW_anosov}.

\begin{definition} \label{def:Anosov}
Let $\Gamma$ be a word hyperbolic group, with boundary at infinity $\partial_{\infty}\Gamma$.
Let $\theta\subset\Delta$ be a nonempty subset of the simple roots with $\theta=\theta^{\star}$.
A representation $\rho : \Gamma\to G$ is \emph{$P_{\theta}$-Anosov} if there exists a $\rho$-equivariant continuous boundary map 
\[ \xi : \partial_\infty \Gamma \to G/P_{\theta} \]
that is dynamics-preserving and transverse and if for any $\alpha\in\theta$,
 \begin{equation} \label{eqn:alpha_i-diverge}
 \lim_{\gamma\to\infty} \alpha(\mu(\rho(\gamma))) = +\infty.
 \end{equation}
\end{definition}

By \eqref{eqn:alpha_i-diverge} we mean that $\lim_{n\to +\infty} \alpha(\mu(\rho(\gamma_n))) = +\infty$ for any sequence $(\gamma_n)_{n\in\NN}$ of distinct elements of~$\Gamma$.
By \emph{dynamics-preserving} we mean that if $\eta$ is the attracting fixed point of some element $\gamma\in\Gamma$ in $\partial_{\infty}\Gamma$, then $\xi(\eta)$ is an attracting fixed point of $\rho(\gamma)$ in $G/P_{\theta}$.
By \emph{transverse} we mean that pairs of distinct points in \(\partial_\infty \Gamma\) are sent to transverse pairs in \(G/P_\theta\), \ie  to pairs belonging to the unique open $G$-orbit in $G/P_{\theta} \times G/P_{\theta}$ (for the diagonal action of~$G$).
The map $\xi$ is unique, entirely determined~by~$\rho$.

By \cite{Labourie_anosov, Guichard_Wienhard_DoD},  any $P_{\theta}$-Anosov representation is a quasi-isometric embedding; in particular, it has a discrete image and a finite kernel.
The set of $P_{\theta}$-Anosov representations is open in $\Hom(\Gamma,G)$.

\subsection{Uniform domination}

Let $\lambda: G \to \overline{\aaa}^+$ be the \emph{Lyapunov projection} of~$G$, \ie the projection induced by the Jordan decomposition: any $g\in G$ can be written uniquely as the commuting product $g = g_h g_e g_u$ of a hyperbolic, an elliptic, and a unipotent element (see \eg \cite[Th.\,2.19.24]{Eberlein}), and $\exp(\lambda(g))$ is the unique element of $\exp(\overline{\aaa}^+)$ in the conjugacy class of~$g_h$.
For any $g\in G$,
\begin{equation} \label{eqn:lambda-limit}
\lambda(g) = \lim_{n \to +\infty} \frac{1}{n} \mu(g^n).
\end{equation}
For any simple root $\alpha\in\Delta$, let $\omega_{\alpha}\in\aaa^{\ast}$ be the fundamental weight associated with~$\alpha$:
\[ 2\frac{\langle{ \omega_{\alpha}, \beta}\rangle}{\langle{\alpha, \alpha}\rangle} = \delta_{\alpha,\beta} \]
for all $\beta\in\Delta$, where $\langle{\cdot,\cdot}\rangle$ is a $W$-invariant inner product on~$\aaa^{\ast}$ and $\delta_{\cdot, \cdot}$ is the Kronecker symbol.
We shall use the following terminology from \cite{GGKW_anosov}.

\begin{definition} \label{def:unif-domin}
A representation $\rho_L : \Gamma\to\Aut_{\KK}(b)$ \emph{uniformly $\omega_{\alpha}$-dominates} a representation $\rho_R : \Gamma\to\Aut_{\KK}(b)$ if there exists $c<1$ such that for any $\gamma\in\Gamma$,
\[ \omega_{\alpha} (\lambda(\rho_R(\gamma))) \leq c \, \omega_{\alpha} (\lambda(\rho_L(\gamma))). \]
\end{definition}

\subsection{Anosov representations into $\Aut_{\KK}(b)$ and $\Aut_{\KK}(b\oplus -b)$}

Let $G=\Aut_{\KK}(b)$ where $b$ is a nondegenerate $\RR$-bilinear form on a $\KK$-vector space~$V$ as in Section~\ref{subsec:intro-compact-G}.

In all cases except when $\KK = \RR$ and $b$ is a symmetric bilinear form of signature $(n,n)$, the restricted root system is of type \(B_n\), \(C_n\), or \(BC_n\). We can choose $\Delta = \{\alpha_i(b) ~|~ 1\leq  i \leq n\}$ so that for any $1\leq i\leq n$ the parabolic subgroup $P_i(b) := P_{\{ \alpha_i(b)\}}$ is the stabilizer of an $i$-dimensional $b$-isotropic $\KK$-subspace of~$V$.
The space $\F_i(b)$ of $i$-dimensional $b$-isotropic $\KK$-subspaces of~$V$ identifies with $G/P_i(b)$.
We have $\alpha_i(b) = \alpha_i(b)^{\star}$ for all $1\leq i\leq n$.

In the case that $\KK=\RR$ and $b$ is a symmetric bilinear form of signature $(n,n)$, the restricted root system is of type \(D_n\). We can still choose $\Delta = \{\alpha_i(b) ~|~ 1\leq  i \leq n\}$ so that for any $1\leq i\leq n-2$ the parabolic subgroup $P_i(b) = P_{\{ \alpha_i(b)\}}$ is the stabilizer of an $i$-dimensional $b$-isotropic subspace of~$V$.
We have $\alpha_i(b) = \alpha_i(b)^{\star}$ for all $1\leq i\leq n-2$.
The parabolic subgroups $P_{n-1}(b) = P_{\{ \alpha_{n-1}(b)\}}$ and $P_n(b) = P_{\{ \alpha_n(b)\}}$ are both stabilizers of $n$-dimensional $b$-isotropic subspaces of~$V$, and are conjugate by some element $g\in\Aut_{\KK}(b)\smallsetminus\Aut_{\KK}(b)_0$.
The stabilizer of an $(n-1)$-dimensional $b$-isotropic subspace is conjugate to $P_{n-1}(b)\cap P_n(b) = P_{\{\alpha_{n-1}(b),\alpha_n(b)\}}$.

We shall use the following result.

\begin{lemma}[{\cite[Th.\,6.3]{GGKW_anosov}}] \label{lem:Ano-unif-domin}
For $\rho_L,\rho_R\in\Hom(\Gamma,\Aut_{\KK}(b))$, the representation $\rho_L \oplus \rho_R: \Gamma\to\Aut_\KK(b) \times \Aut_\KK(-b) \hookrightarrow \Aut_\KK(b\oplus -b)$ is $P_1(b\oplus -b)$-Anosov if and only
if one of the two representations $\rho_L$ or~$\rho_R$ is $P_1(b)$-Anosov and uniformly $\omega_{\alpha_1(b)}$-dominates the other.
\end{lemma}

Since the boundary map of an Anosov representation is dynamics-preserving, Lemma~\ref{lem:Ano-unif-domin} immediately implies the following.

\begin{corollary} \label{cor:boundarymap}
If $\rho_L \oplus \rho_R : \Gamma\to\Aut_\KK(b) \times \Aut_\KK(-b) \hookrightarrow \Aut_\KK(b\oplus -b)$ is $P_1(b\oplus -b)$-Anosov, then its boundary map 
\[ \xi: \partial_\infty \Gamma \longrightarrow \F_1(b\oplus -b) \]
is, up to switching $\rho_L$ and~$\rho_R$, the composition of the boundary map\linebreak $\xi_L :\nolinebreak \partial_{\infty} \Gamma \rightarrow\nolinebreak\F_1(b)$ of~$\rho_L$ with the natural embedding $\F_1(b) \hookrightarrow \F_1(b\oplus -b)$.
\end{corollary}

We will always be able to reduce to $P_1(b)$-Anosov representations into $\Aut_{\KK}(b)$ using the following result. 

\begin{proposition}[{\cite[Prop.\,4.8 \& \S\,6.3]{GGKW_anosov}}] \label{prop:makeP1}
For any real reductive Lie group $G$ and any nonempty subset $\theta\subset\Delta$ of the simple roots, there exist a nondegenerate bilinear form~$b$ on a real vector space~$V$ and an irreducible linear representation $\tau: G \to \Aut_{\RR}(b)$ with the following properties: 
\begin{enumerate}
  \item \label{item:Ano-P-P1} an arbitrary representation $\rho : \Gamma \to G$ is $P_{\theta}$-Anosov if and only if the composition $\tau \circ \rho : \Gamma \to\nolinebreak \Aut_{\RR}(b)$ is $P_1(b)$-Anosov.
  \item \label{item:domin-P-P1} if a representation $\rho_L : \Gamma\to G$ uniformly $\omega_{\alpha}$-dominates another representation $\rho_R : \Gamma\to G$ for all $\alpha\in\theta$, then $\tau\circ\rho_L : \Gamma\to\Aut_{\RR}(b)$ uniformly $\omega_{\alpha_1(b)}$-dominates $\tau\circ\rho_R : \Gamma\to\Aut_{\RR}(b)$.
\end{enumerate}
\end{proposition}

The existence of such $b$ and~$\tau$ satisfying \eqref{item:Ano-P-P1} was first proved in \cite[\S\,4]{Guichard_Wienhard_DoD}.
In fact, the irreducible representations $\tau$ satisfying \eqref{item:Ano-P-P1} and \eqref{item:domin-P-P1} are exactly those for which the highest restricted weight $\chi$ of~$\tau$ satisfies
\[ \{ \alpha\in\Delta \mid \langle\alpha,\chi\rangle > 0\} = \theta \cup \theta^{\star} \]
and for which the weight space corresponding to~$\chi$ is a line; there are infinitely many such~$\tau$.

\begin{example} \label{ex:Ad-GL}
For $G=\GL_n(\RR)$ and $\theta=\{ \varepsilon_1-\varepsilon_2\}$, we can take $\tau$ to be the adjoint representation $\mathrm{Ad} : G\to\GL_{\RR}(\g)$ and $b$ to be the Killing form of~$\g$.
\end{example}

\subsection{Domains of discontinuity} 
We shall use the following result.

\begin{proposition}[{\cite[Th.\,8.6]{Guichard_Wienhard_DoD}}] \label{prop:dod_opq}
Let $\Gamma$ be a word hyperbolic group.
\begin{enumerate}
  \item \label{item:dod1} For any $P_1(b)$-Anosov representation $\rho: \Gamma \to \Aut_\KK(b)$ with boundary map $\xi: \partial_\infty \Gamma \to \F_1(b)$, the group $\Gamma$ acts properly discontinuously and cocompactly, via~$\rho$, on the complement $\Omega$ in $\F_n(b)$ of
\[ \specialK_{\xi} := \bigcup_{\eta \in \partial_\infty \Gamma} \{ W\in \F_n(b) \, |\, \xi(\eta) \subset W\} \subset \F_n(b). \]
  \item \label{item:dod2} Suppose we are \emph{not} in the case that $\KK=\RR$ and $b$ is a symmetric bilinear form of signature $(n,n)$.
  For any $P_n(b)$-Anosov representation $\rho: \Gamma \to\nolinebreak\Aut_\KK(b)$ with boundary map $\xi: \partial_\infty \Gamma \to \F_n(b)$, the group $\Gamma$ acts properly discontinuously and cocompactly, via~$\rho$, on the complement $\Omega$ in $\F_1(b)$ of
\[ \specialK_{\xi} := \bigcup_{\eta \in \partial_\infty \Gamma} \{ \ell\in \F_1(b) \, |\, \ell \subset \xi(\eta) \}  \subset \F_1(b). \]
\end{enumerate}
\end{proposition}
Contrary to what is stated in~\cite[Th.\,8.6]{Guichard_Wienhard_DoD}, the case of \(\OO(n,n)\) (\ie of a restricted root system of type \(D_n\)) has to be excluded in point~\eqref{item:dod2} of the proposition.

\section{Properly discontinuous actions on group manifolds} \label{sec:dod-GxG}

Let $G=\Aut_{\KK}(b)$ where $b$ is a nondegenerate $\RR$-bilinear form on a $\KK$-vector space~$V$ as in Section~\ref{subsec:intro-compact-G}.
By Theorem~\ref{thm:compactif-Aut(b)}, the $(G\times G)$-orbits in the space $\F_N(b\oplus -b)$ of maximal $(b\oplus -b)$-isotropic $\KK$-subspaces of~$V$ are the $\mathcal{U}_i  := \pi^{-1}(\F_i(b) \times \F_i(-b))$, for $0\leq i\leq n$, where
\[ \pi : \F_N(b\oplus -b) \longrightarrow \bigg(\bigcup_{i=0}^n \F_i(b)\bigg) \times \bigg(\bigcup_{i=0}^n \F_i(-b)\bigg) \]
is the projection defined by \eqref{eqn:pi}.
The following generalization of Theorem~\ref{thm:intro_comp_quotient} is an immediate consequence of Theorem~\ref{thm:compactif-Aut(b)}, Corollary~\ref{cor:boundarymap}, and Proposition~\ref{prop:dod_opq}.\eqref{item:dod1}.

\begin{theorem} \label{thm:compactification}
Let $\Gamma$ be a torsion-free word hyperbolic group and $\rho_L,\rho_R : \Gamma\to G=\Aut_{\KK}(b)$ two representations.
Suppose that $\rho_L$ is $P_1(b)$-Anosov and uniformly $\omega_{\alpha_1(b)}$-dominates $\rho_R$ (Definition~\ref{def:unif-domin}).
Then $\Gamma$ acts properly discontinuously, via $\rho_L \oplus \rho_R$, on $(G\times G)/\Diag(G)$.

Let $\xi_L : \partial_{\infty}\Gamma\to\F_1(b)$ be the boundary map of~$\rho_L$.
For any $0\leq i\leq n$, let $\specialK^i_{\xi_L}$ be the subset of $\F_i(b)$ consisting of subspaces $W$ containing $\xi_L(\eta)$ for some $\eta \in \partial_\infty \Gamma$, and let $\mathcal{U}^{\xi_L}_{i}$ be the complement in~$\mathcal{U}_i$ of $\pi^{-1}(\specialK^i_{\xi_L} \times \F_i(-b))$.
Then $\Gamma$ acts properly discontinuously and cocompactly, via $\rho_L \oplus \rho_R$, on the open subset
\[ \Omega := \bigcup_{i=0}^n\, \mathcal{U}_i^{\xi_L} \]
of $\F_N(b\oplus -b)$, and the quotient orbifold $(\rho_L \oplus \rho_R)(\Gamma) \backslash \Omega$ is a compactification of
\[ (\rho_L\oplus\rho_R)(\Gamma) \backslash (G \times G)/\Diag(G). \]
If $\Gamma$ is torsion-free, then this compactification is a smooth manifold.
\end{theorem}

Recall from Lemma~\ref{lem:Ano-unif-domin} that the condition that one of the representations $\rho_L$ or~$\rho_R$ be $P_1(b)$-Anosov and uniformly $\omega_{\alpha_1(b)}$-dominate the other is equivalent to the condition that
\[ \rho := \rho_L \oplus \rho_R : \Gamma \longrightarrow G \times G = \Aut_\KK(b) \times \Aut_\KK(-b) \longhookrightarrow \Aut_\KK(b\oplus -b) \]
be $P_1(b\oplus -b)$-Anosov \cite[Th.\,6.3]{GGKW_anosov}.

\begin{proof}[Proof of Theorem~\ref{thm:compactification}]
By Corollary~\ref{cor:boundarymap}, the boundary map $\xi: \partial_\infty \Gamma \to \F_1(b\oplus -b)$ of $\rho = \rho_L \oplus \rho_R$ is the composition of $\xi_L$ with the natural embedding $\F_1(b) \hookrightarrow \F_1(b\oplus -b)$.
By Proposition~\ref{prop:dod_opq}.\eqref{item:dod1}, the group $\Gamma$ acts properly discontinuously and cocompactly, via~$\rho$, on the open set~$\Omega$.
Note that $\Omega$ contains~$\mathcal{U}_0$, hence the action of $\Gamma$ on~$\mathcal{U}_0$ via~$\rho$ is properly discontinuous.
By Theorem~\ref{thm:compactif-Aut(b)}, the set $\mathcal{U}_0$ is an open and dense $(G \times G)$-orbit in $\F_N(b\oplus -b)$, isomorphic to $(G \times G)/\Diag(G)$.
Therefore, $\Gamma$ acts properly discontinuously via~$\rho$ on $(G\times\nolinebreak G)/\Diag(G)$ and $\rho(\Gamma) \backslash \mathcal{U}_0 \simeq \rho(\Gamma) \backslash (G\times G)/\Diag(G)$ is open and dense in the compact orbifold $\rho(\Gamma) \backslash \Omega$.
This orbifold is a manifold if $\Gamma$ is torsion-free.
\end{proof}

\begin{remark} \label{rem:proper-implies-Ano-rk1}
In the case that $\Aut_{\KK}(b)$ has real rank~$1$, \emph{all} properly discontinuous actions by a quasi-isometric embedding come from a pair of representations $(\rho_L,\rho_R)$ as in Theorem~\ref{thm:compactification}, by \cite[Th.\,6.3]{GGKW_anosov}.
For $\Aut_{\KK}(b)=\OO(2,1)$ we obtain compactifications of \emph{anti-de Sitter} $3$-manifolds, and for $\Aut_{\KK}(b)=\OO(3,1)$ of \emph{holomorphic Riemannian} complex $3$-manifolds of constant nonzero curvature.
We refer to \cite{Goldman_NonStddLortz, Ghys_HomoSL2C, TKobayashi_deformation, Salein, Kassel_these, Gueritaud_Kassel, Gueritaud_Kassel_Wolff, Deroin_Tholozan, Tholozan_domin, Danciger_Gueritaud_Kassel} for examples of such pairs $(\rho_L,\rho_R)$.
\end{remark}

\begin{corollary}[{\cite[Th.\,6.3, (1)\(\Rightarrow\)(6)]{GGKW_anosov}}] \label{cor:higher_rank}
Let $\Gamma$ be a word hyperbolic group, $G$ an arbitrary real reductive Lie group, and $\rho_L,\rho_R : \Gamma\to G$ two representations.
Let $\theta\subset\Delta$ be a nonempty subset of the simple roots of~$G$.
Suppose that $\rho_L$ is $P_{\theta}$-Anosov and uniformly $\omega_{\alpha}$-dominates~$\rho_R$ for all $\alpha\in\theta$.
Then the action of $\Gamma$ on $(G\times G)/\Diag(G)$ via $(\rho_L, \rho_R): \Gamma \to  G\times G$ is properly discontinuous.
\end{corollary}

\begin{proof}
By Proposition~\ref{prop:makeP1}, there exist a nondegenerate bilinear form $b$ on a real vector space~$V$ and a linear representation $\tau: G \to \Aut_{\RR}(b)$ such that $\tau\circ\rho_L : \Gamma\to\Aut_{\RR}(b)$ is $P_1(b)$-Anosov and uniformly $\omega_{\alpha_1(b)}$-dominates $\tau\circ\rho_R$.
By Theorem~\ref{thm:compactification}, the action of~$\Gamma$ on
\[ (\Aut_{\RR}(b) \times \Aut_{\RR}(b))/\Diag(\Aut_{\RR}(b)) \]
via $\tau \circ \rho_L \oplus \tau \circ\nolinebreak \rho_R$ is properly discontinuous.
Since $(\tau(G) \times \tau(G)) / \Diag( \tau(G))$ embeds into $(\Aut_{\RR}(b) \times \Aut_{\RR}(b))/ \Diag(\Aut_{\RR}(b))$ as the $(\tau(G)\times \tau(G))$-orbit of $(e,e)$, the action of $\Gamma$ on $(\tau(G) \times \tau(G)) / \Diag(\tau(G))$ via $\tau\circ\rho_L \oplus \tau\circ\rho_R$ is also properly discontinuous.
Thus the action of $\Gamma$ on $(G\times G)/ \Diag(G)$ via $(\rho_L,\rho_R)$ is properly discontinuous.
\end{proof}

As above, the condition that one of the representations $\tau \circ \rho_L$ or $\tau\circ \rho_R$ be $P_1(b)$-Anosov and uniformly $\omega_{\alpha_1(b)}$-dominate the other is equivalent to the condition that
\[ \tau\circ\rho_L \oplus \tau\circ\rho_R : \Gamma \longrightarrow \Aut_{\KK}(b) \times \Aut_{\KK}(-b) \longhookrightarrow \Aut_{\KK}(b\oplus -b) \]
be $P_1(b\oplus -b)$-Anosov.

\begin{corollary} \label{cor:general_comp}
Let $\Gamma$ be a word hyperbolic group, $G$ an arbitrary real reductive Lie group, and $\rho_L,\rho_R : \Gamma\to G$ two representations of~$\Gamma$.
Let $b$ be a nondegenerate $\RR$-bilinear form on a $\KK$-vector space~$V$ as above, for $\KK=\RR$, $\CC$, or~$\bH$, and let $\tau: G \to \Aut_{\KK}(b)$ be a linear representation of~$G$ such that $\tau\circ\rho_L : \Gamma\to\Aut_{\KK}(b)$ is $P_1(b)$-Anosov and uniformly $\omega_{\alpha_1(b)}$-dominates $\tau\circ\rho_R$ (see Proposition~\ref{prop:makeP1} and Example~\ref{ex:Ad-GL}).
Let $\Omega$ be the cocompact domain of discontinuity of $(\tau\circ\rho_L \oplus \tau\circ\rho_R)(\Gamma)$ in $\F_N(b\oplus -b)$ provided by Proposition~\ref{prop:dod_opq}.\eqref{item:dod1}.
A compactification of
\[ (\tau\circ\rho_L \oplus \tau\circ\rho_R)(\Gamma) \backslash(\tau(G) \times \tau(G)) / \Diag( \tau(G)) \]
is given by its closure in $(\tau\circ\rho_L \oplus \tau\circ\rho_R)(\Gamma) \backslash \Omega$.
If $\tau : G\to\Aut_{\KK}(b)$ has compact kernel, this provides a compactification of $(\rho_L,\rho_R)(\Gamma) \backslash(G \times\nolinebreak G) / \Diag(G)$.
\end{corollary}

In the special case where $\rho_R : \Gamma \to G$ is the trivial representation, the action of $\Gamma$ on $(G\times G)/\Diag(G)$ via $\rho_L \oplus \rho_R$ is the action of $\Gamma$ on~$G$ via left multiplication by $\rho_L$ and Corollary~\ref{cor:general_comp} yields, when $\tau$ has compact kernel, a compactification of $\rho_L(\Gamma)\backslash G \simeq (\rho_L(\Gamma)\times\{ e\})\backslash (G\times G)/\Diag(G)$.

We refer to Theorem~\ref{thm:tame-G/Gamma-LR} for the tameness of $(\rho_L,\rho_R)(\Gamma) \backslash(G \times G)/\Diag(G)$ for general $\rho_L,\rho_R$.

\section{Properly discontinuous actions on other\\ homogeneous spaces} \label{sec:proof-comp-general}

In this section we prove Theorem~\ref{thm:Hpq} and Proposition~\ref{prop:compactification_general}.
We first introduce some notation.
For $\KK=\RR$, $\CC$, or~$\bH$ and $p,q\in\NN$, we denote by $\KK^{p,q}$ the vector space $\KK^{p+q}$ endowed with the $\RR$-bilinear form $b_{\KK}^{p,q}$ of \eqref{eqn:bpq}, so that $\Aut_{\KK}(b_{\KK}^{p,q}) = \OO(p,q)$, $\U(p,q)$, or $\Sp(p,q)$.
We use $b_{\RR\otimes\CC}^{p,q}$ for the complex symmetric bilinear form extending $b_{\RR}^{p,q}$ on $\RR^{p+q} \otimes_{\RR}\CC$, so that $\Aut_{\CC}(b_{\RR\otimes\CC}^{p,q}) = \OO(p+q,\CC)$.
For \(m\in\NN\) and \(\KK=\RR\) or~\(\CC\), we denote by
\[\omega^{2m}_{\KK}:\,(x,y) \longmapsto x_1 y_{m+1} - x_{m+1} y_1 + \dots + x_m y_{2m} - x_{2m} y_m \]
the standard symplectic form on \(\KK^{2m}\), so that $\Aut_{\KK}(\omega_{\KK}^{2m})=\Sp(2m,\KK)$ for $\KK=\RR$ or~$\CC$.
Recall that a Hermitian form $h$ on a $\CC$-vector space~$V$ is completely determined by its real part~$b$: for any $v,v'\in V$,
\[ h(v,v') = b(v,v') - \sqrt{-1} \, b\big(v, \sqrt{-1}v'\big). \]
Similarly, an $\bH$-Hermitian form $h_{\bH}$ on an $\bH$-vector space~$V$ is completely determined by its complex part~$h$: for any $v,v'\in V$,
\[ h_{\bH}(v,v') = h(v,v') - h\big(v,v'j) \, j. \]
Thus an $\bH$-Hermitian form is completely determined by its real part.

\begin{proof}[Proof of Theorem~\ref{thm:Hpq} and of Proposition~\ref{prop:compactification_general} in cases (i), (ii), (iii) of Table~\ref{table2}]
\ (For $\KK=\RR$, see \cite[Prop.\,13.1]{Guichard_Wienhard_DoD}.)
For $p>q>0$, the group $G = \Aut_{\KK}(b_{\KK}^{p,q+1})$ acts transitively on the closed submanifold $\F_1(b_{\KK}^{p,q+1})$ of the smooth compact manifold $\F_1(b_{\KK}^{p+1,q+1})$, which has positive codimension.
The complement $\mathcal{U} = \F_1(b_{\KK}^{p+1,q+1}) \smallsetminus \F_1(b_{\KK}^{p,q+1})$ is open and dense in $\F_1(b_{\KK}^{p+1,q+1})$, and identifies with $\hat{\HH}_{\KK}^{p,q}$ since $G$ acts transitively on~$\mathcal{U}$ and the stabilizer in~$G$ of $[1:0:\ldots:0:1]\in\mathcal{U}\subset\PP(\KK^{p+1,q+1})$ is $H = \Aut_{\KK}(b_{\KK}^{p,q})$.
Thus $\F_1(b_{\KK}^{p+1,q+1})$ is a smooth compactification of $\hat{\HH}_{\KK}^{p,q}$.

Consider a Cartan subspace $\aaa'$ for $G' = \Aut_{\KK}(b_{\KK}^{p+1,q+1})$ that contains a Cartan subspace $\aaa$ for $G = \Aut_{\KK}(b_{\KK}^{p,q+1})$.
If $\KK=\RR$ and $p>q+1$, then $G$ and $G'$ both have restricted root systems of type $B_{q+1}$, hence the restriction of $\alpha_{q+1}(b_{\KK}^{p+1,q+1})$ to $\aaa$ is $\alpha_{q+1}(b_{\KK}^{p,q+1})$.
If $\KK=\CC$ or~$\bH$ and if $p>q+1$, then $G$ and $G'$ both have restricted root systems of type $(BC)_{q+1}$, hence the restriction of $\alpha_{q+1}(b_{\KK}^{p+1,q+1})$ to $\aaa$ is $\alpha_{q+1}(b_{\KK}^{p,q+1})$.
If $\KK=\CC$ or~$\bH$ and if $p=q+1$, then $G$ has a restricted root system of type $C_{q+1}$ and $G'$ of type $(BC)_{q+1}$, hence the restriction of $\alpha_{q+1}(b_{\KK}^{p+1,q+1})$ to $\aaa$ is $\frac{1}{2}\alpha_{q+1}(b_{\KK}^{p,q+1})$.
In all these cases, it follows from Definition~\ref{def:Anosov} that if $\rho: \Gamma\to G = \Aut_{\KK}(b_{\KK}^{p,q+1})$ is a $P_{q+1}(b_{\KK}^{p,q+1})$-Anosov representation with boundary map $\xi : \partial_\infty \Gamma \to \F_{q+1}(b_{\KK}^{p,q+1})$, then the composed representation \(\rho':\Gamma\to G\hookrightarrow G' = \Aut_{\KK}(b_{\KK}^{p+1,q+1})\) is \(P_{q+1}(b_{\KK}^{p+1,q+1})\)-Anosov and that its boundary map $\xi' :\nolinebreak \partial_\infty \Gamma \to\nolinebreak \F_{q+1}(b_{\KK}^{p+1,q+1})$ is the composition of $\xi$ with the natural inclusion $\F_{q+1}(b_{\KK}^{p,q+1}) \hookrightarrow \F_{q+1}(b_{\KK}^{p+1,q+1})$.
By Proposition~\ref{prop:dod_opq}.\eqref{item:dod2}, the group $\Gamma$ acts properly discontinuously and cocompactly, via~$\rho'$, on the complement $\Omega$ in $\F_1(b_{\KK}^{p+1,q+1})$ of
\[ \specialK_{\xi'} = \bigcup_{\eta \in \partial_\infty \Gamma} \{\ell \in \F_{1}(b_{\KK}^{p+1,q+1}) \mid  \ell \subset \xi'(\eta)\}. \]
By construction of~$\xi'$, we have $\specialK_{\xi'} \subset \F_1(b_{\KK}^{p,q+1})$.
Let $\mathcal{C}_\xi$ be the complement of~$\specialK_{\xi'}$ in $\F_1(b_{\KK}^{p,q+1})$.
Then $\Gamma$ acts properly discontinuously and cocompactly, via~$\rho'$, on $\Omega = \mathcal{U}\cup\mathcal{C}_{\xi} \simeq \hat{\HH}_{\KK}^{p,q}\cup\mathcal{C}_{\xi}$.
If $\Gamma$ is torsion-free, then $\rho'(\Gamma) \backslash (\hat{\HH}_{\KK}^{p,q} \cup \mathcal{C}_\xi)$ is a smooth compactification of $\rho(\Gamma) \backslash \hat{\HH}_{\KK}^{p,q}$.

Suppose now that $\KK=\RR$ and $p=q+1$.
Then $G$ has a restricted root system of type $D_{q+1}$ and $G'$ of type $B_{q+1}$, hence the restriction of $\alpha_{q+1}(b_{\KK}^{p+1,q+1})$ to $\aaa$ is $\frac{1}{2} (\alpha_{q+1}(b_{\KK}^{p,q+1}) - \alpha_q(b_{\KK}^{p,q+1}))$.
The boundary map $\xi : \partial_{\infty}\Gamma\to\F_{q+1}(b_{\KK}^{p,q+1})$ of~$\rho$ induces, by composition with the natural inclusion $\F_{q+1}(b_{\KK}^{p,q+1}) \hookrightarrow \F_{q+1}(b_{\KK}^{p+1,q+1})$, a continuous, equivariant, transverse boundary map $\xi' : \partial_\infty \Gamma \to \F_{q+1}(b_{\KK}^{p+1,q+1})$.
If the action of $\Gamma$ on $\hat{\HH}_{\KK}^{p,q}$ via~$\rho$ is properly discontinuous, then the properness criterion of Benoist \cite{Benoist_properness} and Kobayashi \cite{TKobayashi_proper} implies that
\[ \bigl(\alpha_{q+1}(b_{\KK}^{p,q+1}) - \alpha_q(b_{\KK}^{p,q+1})\bigr) (\mu(\rho(\gamma))) \underset{\gamma\to\infty}{\longrightarrow} +\infty \]
and, using \eqref{eqn:lambda-limit}, that $\xi'$ is dynamics-preserving.
Therefore, the composed representation \(\rho' : \Gamma\to G\hookrightarrow G' = \Aut_{\KK}(b_{\KK}^{p+1,q+1})\) is \(P_{q+1}(b_{\KK}^{p+1,q+1})\)-Anosov and we conclude as above.
\end{proof}

\begin{remark}
Identifying $\RR^{2n+2}$ with $\CC^{n+1}$ gives a $\U(n,1)$-equivariant identification~of $\hat{\HH}_{\RR}^{2n,2}$ with $\hat{\HH}_{\CC}^{n,1}$.\,Examples of $P_2(b_{\RR}^{2n,2})$-Anosov representations $\rho :\nolinebreak\Gamma\to\nolinebreak\OO(2n,2)$ inclu\-de the composition of any convex cocompact representation $\rho_1 :\nolinebreak\Gamma\to\nolinebreak\U(n,1)$ with the natural inclusion of $\U(n,1)$ into $\OO(2n,2)$; the manifold $\rho(\Gamma)\backslash\hat{\HH}_{\RR}^{2n,2}$ then identifies with $\rho_1(\Gamma)\backslash\hat{\HH}_{\CC}^{n,1}$, and the compactifications of these two manifolds given by Theo\-rem~\ref{thm:Hpq}.\eqref{item:compactif-quotient-Hpq} coincide.
The same holds if $(\hat{\HH}_{\RR}^{2n,2}, \hat{\HH}_{\CC}^{n,1}, \OO(2n,2), \U(n,1), P_2(b_{\RR}^{2n,2}))$ is replaced by $(\hat{\HH}_{\RR}^{4n,4}, \hat{\HH}_{\bH}^{n,1}, \OO(4n,4), \Sp(n,1), P_4(b_{\RR}^{4n,4}))$ or $(\hat{\HH}_{\CC}^{2n,2}, \hat{\HH}_{\bH}^{n,1}, \U(2n,2),\linebreak\Sp(n,1), P_2(b_{\CC}^{2n,2}))$.
\end{remark}

The following example shows that if $\KK=\RR$ and $p=q+1$, then the fact that $\rho : \Gamma\to G = \Aut_{\KK}(b_{\KK}^{p,q+1})$ is $P_{q+1}(b_{\KK}^{p,q+1})$-Anosov does not imply that the action of $\Gamma$ on $\hat{\HH}_{\KK}^{p,q}$ via~$\rho$ is properly discontinuous.

\begin{example} \label{ex:O(2,2)}
Let $\KK=\RR$ and $p=q+1=2$.
Then the identity component $G_0$ of $G=\OO(2,2)$ identifies with $\PSL_2(\RR)\times\PSL_2(\RR)$ and $\hat{\HH}_{\RR}^{2,2}$ is a covering of order two of $(\PSL_2(\RR)\times\PSL_2(\RR))/\Diag(\PSL_2(\RR))$.
A representation $\rho : \Gamma\to G_0$ is $P_2(b_{\RR}^{2,2})$-Anosov if and only if the projection of~$\rho$ to the first (or second, depending on the numbering of the simple roots) $\PSL_2(\RR)$ factor is convex cocompact.
However, the action of $\Gamma$ via $\rho$ is properly discontinuous on $\hat{\HH}_{\RR}^{2,2}$ if and only if the projection of $\rho$ to one $\PSL_2(\RR)$ factor is convex cocompact and uniformly dominates the other, by \cite[Th.\,6.3]{GGKW_anosov} (see Remark~\ref{rem:proper-implies-Ano-rk1}).
\end{example}

\begin{proof}[Proof of Proposition~\ref{prop:compactification_general} in case~(iv) of Table~\ref{table2}]
(For $q=1$, see \cite[Th.\,13.3]{Guichard_Wienhard_DoD}.) 
We identify $\CC^{p+q}$ with $\RR^{2p+2q}$ and see $H = \Aut_{\CC}(b_{\CC}^{p,q})$ as the subgroup of $G = \Aut_{\RR}(b_{\RR}^{2p,2q})$ commuting with the multiplication by \(\sqrt{-1}\), which we denote by $I\in G$.
The group $G$ acts transitively on the open set 
\[ \mathcal{U} := \{ W'\in \F_{p+q}(b_{\RR\otimes\CC}^{2p,2q}) \mid W' \cap \RR^{2p+2q} = \{0\}\}, \]
and the stabilizer in~$G$ of
\[ W'_0 := \big\{ x + \sqrt{-1}I(x) \mid x \in \RR^{2p+2q} \big\} \in \mathcal{U} \]
is~$H$.
Thus $\mathcal{U}$ identifies with $G/H$ and the closure $\overline{\mathcal{U}}$ of $\mathcal{U}$ in $\F_{p+q}(b_{\RR\otimes\CC}^{2p,2q}) \simeq G'/P'$ provides a compactification of $G/H$.

If \(\rho: \Gamma\to G\) is \(P_1(b_{\RR}^{2p,2q})\)-Anosov, with boundary map $\xi : \partial_\infty \Gamma \to \F_{1}(b_{\RR}^{2p,2q})$, then it easily follows from Definition~\ref{def:Anosov} that the composed representation \(\rho' : \Gamma\to G\hookrightarrow G' = \Aut_{\CC}(b^{2p,2q}_{\RR\otimes\CC})\) is $P_1(b_{\RR\otimes\CC}^{2p,2q})$-Anosov and that its boundary map $\xi' :\partial_\infty \Gamma \to \F_{1}(b_{\RR\otimes\CC}^{2p,2q})$ is the composition of~$\xi$ with the natural inclusion $\F_1(b_{\RR}^{2p,2q}) \hookrightarrow \F_1(b_{\RR\otimes\CC}^{2p,2q})$.
By Proposition~\ref{prop:dod_opq}.\eqref{item:dod1}, the group $\Gamma$ acts properly discontinuously and cocompactly, via~$\rho'$, on the complement $\Omega$ in \(\F_{p+q}(b_{\RR\otimes\CC}^{2p,2q})\) of 
\[ \specialK_{\xi'} = \bigcup_{\eta \in \partial_\infty \Gamma} \big\{ W' \in \F_{p+q}(b_{\RR\otimes\CC}^{2p,2q}) \mid  \xi'(\eta)\subset W'\big\} . \]
Note that $\Omega$ contains~$\mathcal{U}$, hence $\Gamma$ acts properly discontinuously on $G/H$ via~$\rho$ and the quotient $\rho'(\Gamma)\backslash (\Omega\cap\overline{\mathcal{U}})$ provides a compactification of $\rho(\Gamma)\backslash G/H$.
\end{proof}

\begin{proof}[Proof of Proposition~\ref{prop:compactification_general} in case~(v) of Table~\ref{table2}]
Let us write $\bH = \CC + \CC j$ where $j^2=-1$.
We identify $\bH^{p+q}$ with $\CC^{2p+2q}$ and see $H = \Aut_{\bH}(b_{\bH}^{p,q})$ as the subgroup of $G = \Aut_{\CC}(b_{\CC}^{2p,2q})$ commuting with the right multiplication by~$j$, which we denote by $J \in G$.
The tensor product \(\CC^{2p+2q} \otimes_\CC \bH\) can be realized as the set of ``formal'' sums
\[ \CC^{2p+2q} \otimes_\CC \bH =\{ v_1 + v_2 j \mid v_1,\, v_2 \in \CC^{2p+2q}\}, \]
on which \(\bH\) acts by right multiplication: \( (v_1+ v_2 j)\cdot z = z v_1 + \overline{z} v_2 j\) for $z\in\CC$ and \(( v_1 + v_2 j)\cdot j = -v_2 + v_1 j\).
Consider the $\CC$-Hermitian form $h$ on \(\CC^{2p+2q} \otimes_\CC \bH\) given by
\[ h( v_1 + v_2 j, v'_1 + v'_2 j) = b_{\CC}^{2p,2q}(v_1, v'_1) - \overline{b_{\CC}^{2p,2q}(v_2,v'_2)}, \]
and let $h_{\bH}$ be the $\bH$-Hermitian form on \(\CC^{2p+2q} \otimes_\CC \bH\) with complex part~$h$.
Then \(G' = \Sp(p+q,p+q)\) identifies with $\Aut_{\bH}(h_{\bH})$, and the natural embedding of $G = \Aut_{\CC}(b_{\CC}^{2p,2q})$ into~\(G'\) induces a natural embedding of \(\F_{1}(b_{\CC}^{2p,2q})\) into \(\F_{1}(h_\bH)\), given explicitly by
\[ \ell \longmapsto \{ v_1 + v_2 j \mid v_1, v_2 \in \ell\} \]
where $\ell$ is a $b_{\CC}^{2p,2q}$-isotropic line of~$\CC^{2p+2q}$.
The group $G$ acts transitively on the open set
\[ \mathcal{U} := \{ W'\in \F_{p+q}(h_{\bH}) \mid W'\cap \CC^{2p+2q} = \{0\}\}, \]
and the stabilizer in~$G$ of
\[ W'_0 := \{ v + (Jv)j \mid v\in \CC^{2p,2q}\} \in \mathcal{U} \]
is~$H$.
Thus $\mathcal{U}$ identifies with $G/H$ and the closure $\overline{\mathcal{U}}$ of $\mathcal{U}$ in $\F_{p+q}(h_{\bH})\simeq G'/P'$ provides a compactification of $G/H$.

If \(\rho: \Gamma \to G\) is \(P_1(b_{\CC}^{2p,2q})\)-Anosov, with boundary map $\xi : \partial_\infty \Gamma \to \F_{1}(b_{\CC}^{2p,2q})$, then it easily follows from Definition~\ref{def:Anosov} that the composed representation \(\rho': \Gamma \to G \hookrightarrow G'\) is \(P_1(h_\bH)\)-Anosov and that its boundary map $\xi' : \partial_\infty \Gamma \to \F_{1}(h_{\bH})$ is the composition of~$\xi$ with the natural inclusion $\F_1(b_{\CC}^{2p,2q}) \hookrightarrow \F_1(h_{\bH})$.
By Proposition~\ref{prop:dod_opq}.\eqref{item:dod1}, the group $\Gamma$ acts properly discontinuously and cocompactly, via~$\rho'$, on the complement $\Omega$ in \(\F_{p+q}(b_\bH)\) of
\[\specialK_{\xi'} = \bigcup_{\eta \in \partial_\infty \Gamma} \{ W' \in \F_{p+q}(b_\bH) \mid  \xi'(\eta) \subset W'\}. \]
Note that $\Omega$ contains~$\mathcal{U}$, hence $\Gamma$ acts properly discontinuously on $G/H$ via~$\rho$ and the quotient $\rho'(\Gamma)\backslash (\Omega\cap\overline{\mathcal{U}})$ provides a compactification of $\rho(\Gamma)\backslash G/H$.
\end{proof}

\begin{proof}[Proof of Proposition~\ref{prop:compactification_general} in case~(vi) of Table~\ref{table2}]
(For $p=0$, see \cite[\S\,12]{Guichard_Wienhard_DoD}.) 
Let $(e_1,\dots,e_{2m})$ be the standard basis of~$\CC^{2m}$.
The Hermitian form \(h (v,w) = i \omega^{2m}_{\CC}(\overline{v}, w)\) has signature \( (m,m)\).
For any $0\leq p\leq m$, the $\CC$-span of $e_1 - i e_{m+1}, \dots$, $e_p - i e_{m+p}, e_{p+1} + i e_{m+p+1}, \dots, e_m + i e_{2m}$ defines an element $W'_0 \in \F_m(\omega^{2m}_{\CC})$ (a Lagrangian of~$\CC^{2m}$) such that the restriction of $h$ to $W'_0\times W'_0$ is nondegenerate with signature $(p,m-p)$.
The stabilizer of $W'_0$ in $G = \Aut_{\RR}(\omega^{2m}_{\RR})$ is $H = \Aut_{\CC}(b_{\CC}^{p,m-p})$.
Thus the $G$-orbit $\mathcal{U}$ of~$W'_0$ identifies with $G/H$ and the closure $\overline{\mathcal{U}}$ of $\mathcal{U}$ in $\F_m(\omega^{2m}_{\CC})\simeq G'/P'$ provides a compactification of $G/H$.

If \( \rho: \Gamma \to G = \Aut_{\RR}(\omega^{2m}_{\RR})\) is $P_1(\omega^{2m}_{\RR})$-Anosov, with boundary map $\xi : \partial_\infty \Gamma \to \F_{1}(\omega^{2m}_{\RR}) \simeq \RR\PP^{2m-1}$, then it easily follows from Definition~\ref{def:Anosov} that the composed representation \(\rho' : \Gamma \to G \hookrightarrow G' = \Aut_{\CC}(\omega^{2m}_{\CC})\) is $P_1(\omega^{2m}_{\CC})$-Anosov and that its boundary map $\xi' : \partial_\infty \Gamma \to \F_{1}(\omega^{2m}_{\CC}) \simeq \CC\PP^{2m-1}$ is the composition of~$\xi$ with the natural inclusion $\RR\PP^{2m-1} \hookrightarrow \CC\PP^{2m-1}$.
By Proposition~\ref{prop:dod_opq}.\eqref{item:dod1}, the group $\Gamma$ acts properly discontinuously and cocompactly, via~$\rho'$, on the complement $\Omega$~in~\(\F_{m}(\omega^{2m}_{\CC})\)~of
\[ \specialK_{\xi'} = \bigcup_{\eta \in \partial_\infty \Gamma} \{ W' \in \F_m(\omega^{2m}_{\CC}) \, |\,  \xi'(\eta) \subset W'\}. \]
Note that $\Omega$ contains the $G$-invariant open set \[ \mathcal{U}' = \{ W' \in \F_m(\omega^{2m}_{\CC}) \, |\, W' \cap \RR^{2m} = \{ 0\}\}, \]
which itself contains $W'_0$, hence~$\mathcal{U}$.
Thus $\Gamma$ acts properly discontinuously on $G/H$ via~$\rho$ and the quotient $\rho'(\Gamma)\backslash (\Omega\cap\overline{\mathcal{U}})$ provides a compactification of $\rho(\Gamma)\backslash G/H$.
\end{proof}

We now use Remark~\ref{rem:lift-DoD} to compactify other reductive homogeneous spaces that are not affine symmetric spaces.

\begin{proposition} \label{prop:compactification_general_nonsym}
Let $(G, H, P, G', P', P'')$ be as in Table~\ref{table3}.
\begin{enumerate}
  \item \label{item:compactif_general_G/H} There exists an open $G$-orbit $\mathcal{U}$ in $G'/P''$ that is diffeomorphic to $G/H$; the closure $\overline{\mathcal{U}}$ of $\mathcal{U}$ in $G'/P''$ provides a compactification of $G/H$.
  \item For any word hyperbolic group~$\Gamma$ and any $P$-Anosov representation $\rho: \Gamma \to G$, the cocompact domain of discontinuity $\Omega\subset G'/P'$ for $\rho(\Gamma)$ constructed in \cite{Guichard_Wienhard_DoD} (see Proposition~\ref{prop:dod_opq}) lifts to a cocompact domain of discontinuity $\tilde{\Omega}\subset G'/P''$ that contains~$\mathcal{U}$; the quotient $\rho(\Gamma)\backslash (\tilde{\Omega}\cap\overline{\mathcal{U}})$ provides a compactification of $\rho(\Gamma)\backslash G/H$.
\end{enumerate}
\end{proposition}

\begin{table}[h!]
\centering
\begin{adjustwidth}{-1.2in}{-1.2in}
\centering
\begin{tabular}{|c|c|c|c|c|c|c|}
\hline
& $G$ & $H$ & $P$ & $G'$ & $P'$ & $P''$ \tabularnewline
\hline
(vii) & $\OO(4p,4q)$ & $\Sp(2p,2q)$ & $\mathrm{Stab}_G(\ell)$ & $\Sp(2p+2q,2p+2q)$ & $\mathrm{Stab}_{G'}(W')$ & $\mathrm{Stab}_{G'}(W''\subset W')$\tabularnewline
(viii) & $\Sp(4m,\RR)$ & $\OO^*(2m)$ & \(\mathrm{Stab}_G(\ell)\) & $\OO^*(8m)$ & $\mathrm{Stab}_{G'}(W')$ & \(\mathrm{Stab}_{G'}(W''\subset W')\)\tabularnewline
\hline
\end{tabular}
\end{adjustwidth}
\vspace{0.2cm}
\caption{Reductive groups $H\subset G\subset G'$ and parabolic subgroups $P$ of~$G$ and $P'\supset P''$ of~$G'$ to which Proposition~\ref{prop:compactification_general_nonsym} applies.
Here $m,p,q$ are any positive integers.
We denote by $\ell$ an isotropic line and by $W'$ a maximal isotropic subspace (over $\RR$ or~$\bH$), relative to the form $b$ preserved by $G$ or~$G'$.
We also denote by $(W'' \subset W')$ a partial flag of two isotropic subspaces with $W'$ maximal and $\dim_{\RR}(W')=2\,\dim_{\RR}(W'')$.}
\label{table3}
\end{table}

\begin{proof}[Proof of Proposition~\ref{prop:compactification_general_nonsym} in case~(vii) of Table~\ref{table3}]
Let us write $\bH = \RR + \RR i + \RR j + \RR k$ where $i=\sqrt{-1}$ and $ij=k$.
We identify $\bH^{p+q}$ with $\RR^{4p+4q}$, and see $H = \Aut_{\bH}(b_{\bH}^{p,q})$ as the subgroup of $G = \Aut_{\RR}(b_{\RR}^{4p,4q})$ commuting with the right multiplications by $i$ and by~$j$, which we denote respectively by $I,J \in G$.
The tensor product $\RR^{4p+4q} \otimes_\RR \bH$ can be realized as the set of ``formal'' sums
\[ \RR^{4p+4q} \otimes_\RR \bH = \big\{ v_1 + v_2 i + v_3 j + v_4 k \mid v_1, v_2, v_3, v_4 \in \RR^{4p+4q}\big\} . \]
Consider the real bilinear form $b$ on $\RR^{4p+4q} \otimes_\RR \bH$ given by
\[ b(v_\bH, v'_\bH) = b_{\RR}^{4p,4q}(v_1,v'_1) - b_{\RR}^{4p,4q}(v_2,v'_2) + b_{\RR}^{4p,4q}(v_3,v'_3) - b_{\RR}^{4p,4q}(v_4,v'_4) \]
for any $v_\bH = v_1 + v_2 i + v_3 j + v_4 k$ and $ v'_\bH = v'_1 + v'_2 i + v'_3 j + v'_4 k$ in $\RR^{4p+4q} \otimes_\RR \bH$, and let $b_{\bH}$ be the $\bH$-Hermitian form on $\RR^{4p+4q} \otimes_\RR \bH$ with real form~$b$.
Then $G' = \Sp(2p+2q,2p+2q)$ identifies with $\Aut_{\bH}(b_{\bH})$, and the natural embedding of $G = \Aut_{\RR}(b_{\RR}^{4p,4q})$ into~\(G'\) induces a natural embedding of \(\F_{1}(b_{\RR}^{4p,4q})\) into \(\F_{1}(b_\bH)\).

Let $\F_{p+q,2p+2q}(b_{\bH})$ be the space of partial flags $(W'' \subset W')$ of $\RR^{4p+4q} \otimes_\RR \bH$ with $W' \in \F_{2p+2q}(b_{\bH})$ and $\dim_{\bH}(W') = 2 \dim_{\bH}(W'')$; it identifies with $G'/P''$ and fibers $G'$-equivariantly over $\F_{2p+2q}(b_{\bH})\simeq G'/P'$ with compact fiber.
Consider the element $(W''_0 \subset W'_0) \in \F_{p+q,2p+2q}(b_{\bH})$ given by
\[ \left\{ \begin{array}{ccl}
  W''_0 & := & \{ v + (I v) i + (J v) j + (K v) k \mid v \in \RR^{4p+4q}\},\\
  W'_0 & := & \{ v + (I v) i + (J v') j + (K v') k \mid v,v' \in \RR^{4p+4q}\}.
\end{array} \right. \]
Its stabilizer in $G = \Aut_{\RR}(b_{\RR}^{4p,4q})$ is the set of elements \(g\) commuting with $I$ and~$J$, namely \(H = \Aut_{\bH}(b_{\bH}^{p,q})\).
Thus the $G$-orbit $\mathcal{U}$ of $(W''_0 \subset W'_0)$ in $\F_{p+q,2p+2q}(b_{\bH})$ identifies with $G/H$ and the closure $\overline{\mathcal{U}}$ of $\mathcal{U}$ in $\F_{p+q,2p+2q}(b_{\bH})\simeq G/P''$ provides a compactification of $G/H$.

If \(\rho: \Gamma \to G\) is \(P_1(b_{\RR}^{4p,4q})\)-Anosov, with boundary map $\xi : \partial_\infty \Gamma \to \F_{1}(b_{\RR}^{4p,4q})$, then it easily follows from Definition~\ref{def:Anosov} that the composed representation \(\rho': \Gamma \to G \hookrightarrow G'\) is \(P_1(b_\bH)\)-Anosov and that its boundary map $\xi' : \partial_\infty \Gamma \to \F_{1}(b_{\bH})$ is the composition of~$\xi$ with the natural inclusion $\F_1(b_{\RR}^{4p,4q}) \hookrightarrow \F_1(b_{\bH})$.
By Proposition~\ref{prop:dod_opq}.\eqref{item:dod1}, the group $\Gamma$ acts properly discontinuously and cocompactly, via~$\rho'$, on the complement $\Omega$ in \(\F_{2p+2q}(b_\bH)\) of
\[\specialK_{\xi'} = \bigcup_{\eta \in \partial_\infty \Gamma} \{ W' \in \F_{2p+2q}(b_\bH) \mid \xi'(\eta) \subset W'\}. \]
Since $\F_{p+q,2p+2q}(b_{\bH})$ fibers $G'$-equivariantly over $\F_{2p+2q}(b_{\bH})$ with compact fiber, $\Gamma$ also acts properly discontinuously, via~$\rho'$, on the preimage $\tilde{\Omega}$ of $\Omega$ in $\F_{p+q,2p+2q}(b_{\bH})$.
One checks that $\tilde{\Omega}$ contains the $G$-invariant open set
\[ \mathcal{U}' := \{ (W'' \subset W') \in \F_{p+q,2p+2q}(b_{\bH}) \mid W' \cap \RR^{4p+4q} = \{0\}\}, \]
which itself contains \((W''_0 \subset W'_0)\), hence~$\mathcal{U}$.
Thus $\Gamma$ acts properly discontinuously on $G/H$ via~$\rho$ and the quotient $\rho'(\Gamma)\backslash (\tilde{\Omega}\cap\overline{\mathcal{U}})$ provides a compactification of $\rho(\Gamma)\backslash G/H$.
\end{proof}

Case~(viii) of Table~\ref{table3} is similar to case~(vii): just replace the real quadratic form \(b_{\RR}^{4p,4q}\) on~$\RR^{4p+4q}$ with the symplectic form $\omega^{4m}_{\RR}$ on~$\RR^{4m}$, and \(b\) with the symplectic form $\omega^{4m}_{\RR}(v_1,v'_1) - \omega^{4m}_{\RR}(v_2,v'_2) + \omega^{4m}_{\RR}(v_3,v'_3) - \omega^{4m}_{\RR}(v_4,v'_4)$ on $\RR^{4m} \otimes_{\RR} \bH$.
The subgroup of \(G = \Aut_{\RR}(\omega_{\RR}^{4m})\) commuting with \(I\) and~\(J\) is \(H = \OO^*(2m)\).

\section{Topological tameness} \label{sec:tameness}

Lemma~\ref{lem:tameness_intro} is a particular case of the following general principle.
Recall that a \emph{real semi-algebraic set} is a set defined by real polynomial equalities and inequalities.

\begin{proposition}\label{prop:semialgebraic}
  Let \(X\) be a real semi-algebraic set and \(\Gamma\) a torsion-free discrete group acting on~$X$ by real algebraic homeomorphisms.
  Suppose \(\Gamma\) acts properly discontinuously and cocompactly on some open subset \(\Omega\) of~\(X\).
  Let \(\mathcal{U}\) be a \(\Gamma\)-invariant real semi-algebraic subset of~$X$ contained in \(\Omega\) (\eg an orbit of a real algebraic group containing \(\Gamma\) and acting algebraically on~\(X\)).
Then the closure \(\overline{\mathcal{U}}\) of $\mathcal{U}$ in~$X$ is real semi-algebraic and \(\Gamma
  \backslash (\overline{\mathcal{U}} \cap \Omega)\) is compact and has a triangulation such that \(\Gamma \backslash (\partial \mathcal{U} \cap \Omega)\) is a finite union of simplices.
If \(\mathcal{U}\) is a manifold, then \(\Gamma \backslash \mathcal{U}\) is topologically tame.
\end{proposition}

We use the notation $\mathring{D}$ for the interior of a subset $D$ of~$X$ and $\partial D=\overline{D}\smallsetminus\mathring{D}$ for its boundary.

\begin{proof}
The fact that \(\overline{\mathcal{U}}\) is itself real semi-algebraic is classical, see \cite[Cor.\,2.5]{Coste_SAG}.
It is also a general principle that orbits are real semi-algebraic, this follows \eg from \cite[Cor.\,2.4.(2)]{Coste_SAG}.
Consequently, every point in the open subset \(\overline{\mathcal{U}} \cap \Omega\) has a real semi-algebraic neighborhood, and the same property holds in the quotient space \(\Gamma \backslash (\overline{\mathcal{U}} \cap \Omega)\).
Thus \(\Gamma \backslash (\overline{\mathcal{U}} \cap \Omega)\) is a finite union of closed semi-analytic subsets \(D_1, \dots, D_n\), each algebraically homeomorphic to a semi-algebraic subset of a Euclidean space.

Applying \cite[Th.~3.12]{Coste_SAG}, we deduce that there is a triangulation of \(D_1\) such that the subsets \(D_1 \cap \Gamma \backslash (\partial \mathcal{U} \cap \Omega)\), \(\partial D_1 \cap D_2\), \dots{}, \(\partial D_1 \cap D_n\) are all finite unions of simplices.
  Similarly, there is a triangulation of \(D_2 \smallsetminus \mathring{D_1}\) such that the subsets \(\partial D_1 \cap D_2\), \((D_2 \smallsetminus \mathring{D_1}) \cap \Gamma \backslash (\partial \mathcal{U} \cap \Omega)\), \(\partial (D_2 \smallsetminus \mathring{D_1}) \cap D_3\), \dots{}, \(\partial (D_2 \smallsetminus \mathring{D_1}) \cap D_n\) are all unions of simplices.
By taking a refinement of these two triangulations, we obtain a triangulation of \(D_1 \cup D_2\) such that the subsets \(( D_1 \cup D_2) \cap \Gamma \backslash (\partial \mathcal{U} \cap \Omega)\), \(\partial ( D_1 \cup D_2) \cap D_3\), \dots{}, \(\partial ( D_1 \cup D_2) \cap D_n\) are all finite unions of simplices.
By induction, for any $1\leq k\leq n$ there is a triangulation of \(E_k := D_1 \cup \cdots \cup D_k\) such that the subsets \( E_k \cap \Gamma \backslash (\partial \mathcal{U} \cap \Omega)\), \(\partial E_k \cap D_{k+1}\), \dots{}, \(\partial E_k \cap D_n\) are all finite union of simplices.
For \(k=n\), this gives a triangulation of \(\Gamma \backslash (\overline{\mathcal{U}} \cap \Omega)\) such that \(\Gamma \backslash (\partial \mathcal{U} \cap \Omega)\) is a finite union of simplices.

This triangulation allows us to build a tubular neighborhood of \(\Gamma \backslash (\partial\mathcal{U} \cap \Omega)\) such that  \(\Gamma \backslash \mathcal{U}\) is homeomorphic to the complement of this tubular neighborhood. 
Thus, if \(\mathcal{U}\) is a manifold, then \(\Gamma \backslash \mathcal{U}\) is homeomorphic to the interior of a compact manifold with boundary.
\end{proof}

From Theorem~\ref{thm:intro_comp_quotient} and Lemma~\ref{lem:tameness_intro} we deduce the following.
Theorem~\ref{thm:tame-G/Gamma} corresponds to the special case where $\rho_R$ is constant.

\begin{theorem}\label{thm:tame-G/Gamma-LR}
Let $\Gamma$ be a torsion-free word hyperbolic group, $G$ a real reductive Lie group, and $\rho_L,\rho_R : \Gamma\to G$ two representations.
Let $\theta\subset\Delta$ be a nonempty subset of the simple roots of~$G$.
Suppose that $\rho_L$ is $P_{\theta}$-Anosov and uniformly $\omega_{\alpha}$-dominates~$\rho_R$ for all $\alpha\in\theta$.
Then $(\rho_L, \rho_R)(\Gamma)\backslash (G\times G)/\Diag(G)$ is a topologically tame manifold.
\end{theorem}

For $G=\SO(p,1)$, this was first proved in \cite[Th.\,1.8 \& Prop.\,7.2]{Gueritaud_Kassel}.
In that case, tameness actually still holds when $\rho_L$ is allowed to be geometrically finite instead of convex cocompact.

\begin{proof}[Proof of Theorem~\ref{thm:tame-G/Gamma-LR}]
By Proposition~\ref{prop:makeP1}, there exist a nondegenerate bilinear form $b$ on a real vector space~$V$ and a linear representation $\tau : G \to \Aut_{\RR}(b)$ such that $\tau\circ\rho_L : \Gamma\to\Aut_{\RR}(b)$ is $P_1(b)$-Anosov and uniformly $\omega_{\alpha_1(b)}$-dominates $\tau\circ\rho_R$.
Let $\Omega$ be the cocompact domain of discontinuity of $(\tau\circ\rho_L \oplus \tau\circ\rho_R)(\Gamma)$ in $\F_N(b\oplus -b)$ given by Proposition~\ref{prop:dod_opq}.\eqref{item:dod1}.
By Theorem~\ref{thm:intro_comp_quotient}, it contains the open $(\Aut_{\RR}(b)\times\Aut_{\RR}(b))$-orbit $\mathcal{U}_0$ of Theorem~\ref{thm:compactif-Aut(b)}, which identifies with $(\Aut_{\RR}(b)\times\Aut_{\RR}(b))/\Diag(\Aut_{\RR}(b))$.
Let \(u\) be a point in~$\mathcal{U}_0$ with stabilizer equal to \(\Diag(\Aut_{\RR}(b))\).
Applying Lemma~\ref{lem:tameness_intro} to the $(\tau\oplus\tau)(G)$-orbit $\mathcal{U}$ of $u$ in~$\mathcal{U}_0$, we see that $(\tau\circ\rho_L \oplus \tau\circ\rho_R)(\Gamma)\backslash (\tau(G) \times \tau(G)) / \Diag(\tau(G))$ is a topologically tame manifold.
If $\tau$ has finite kernel, then $(\rho_L\oplus\rho_R)(\Gamma)\backslash (G\times G)/G$ is a topologically tame manifold as well.

However, in general $\tau$ might not have finite kernel.
To address this issue, we force injectivity by introducing another representation, as follows.
Let $\tau' : G \to \GL_{\RR}(V')$ be any injective linear representation of~\(G\) where $V'$ is a real vector space of dimension $N'\in\NN$.
The Grassmannian $\F_{N'}(V' \oplus\nolinebreak V')$ is compact, hence the action of \(\Gamma\) on \(\Omega \times \F_{N'}(V' \oplus\nolinebreak V')\) via
\[ (\tau\circ\rho_L \oplus \tau\circ\rho_R) \times (\tau'\circ\rho_L \oplus \tau'\circ\rho_R) \]
is properly discontinuous and cocompact.
By Theorem~\ref{thm:compactif-GL}, there is an open\linebreak $(\GL_{\RR}(V')\times\nolinebreak\GL_{\RR}(V'))$-orbit $\mathcal{U}'_0$ in $\mathcal{F}_{N'}(V' \oplus V')$ that identifies with
\[ (\GL_{\RR}(V')\times\GL_{\RR}(V'))/\Diag(\GL_{\RR}(V')). \]
Let $u'$ be a point in~$\mathcal{U}'_0$ with stabilizer \(\Diag(\GL_{\RR}(V'))\) in $\GL_{\RR}(V')\times\GL_{\RR}(V')$.
By injectivity of~$\tau'$, the stabilizer of $(u,u')$ in $G\times G$ for the action of $G\times G$ on $\F_N(V\oplus V) \times \F_{N'}(V'\oplus V')$ via $(\tau\oplus\tau) \times (\tau'\oplus \tau')$ is \(\Diag(G)\).
Applying Lemma~\ref{lem:tameness_intro} to the $((\tau\oplus\tau) \times (\tau'\oplus \tau'))(G)$-orbit $\mathcal{U}$ of $(u,u')$ and to $\Omega \times \F_{N'}(V' \oplus V')$ instead of~$\Omega$, we obtain that $(\rho_L, \rho_R)(\Gamma)\backslash (G\times G)/\Diag(G)$ is a topologically tame manifold.
\end{proof}


\end{document}